\documentclass[11pt]{amsart}

\input{LPPBuseProcMacros}

\author[C.~Janjigian]{Christopher Janjigian}
\address{Christopher Janjigian\\ Purdue University\\  Department of Mathematics \\ 150 N University St\\ West Lafayette, IN 47901\\ USA.}
\email{cjanjigi@purdue.edu}
\urladdr{http://www.math.purdue.edu/~cjanjigi}
\thanks{C~Janjigian was partially supported by National Science Foundation grant DMS-2125961 and Simons Foundation grant MPS-TSM-00012155.}

\author[F.~Rassoul-Agha]{Firas Rassoul-Agha}
\address{Firas Rassoul-Agha\\ University of Utah\\  Department of Mathematics\\ 155S 1400E\\   Salt Lake City, UT 84112\\ USA.}
\email{firas@math.utah.edu}
\urladdr{http://www.math.utah.edu/~firas}
\thanks{F.\ Rassoul-Agha was partially supported by National Science Foundation grants DMS-2054630 and DMS-2450951.}

\author[T.~Sepp\"al\"ainen]{Timo Sepp\"al\"ainen}
\address{Timo Sepp\"al\"ainen\\ University of Wisconsin--Madison\\  Department of Mathematics\\ 480 Lincoln Drive \\  Madison, WI 53706\\ USA.}
\email{seppalai@math.wisc.edu}
\urladdr{http://www.math.wisc.edu/~seppalai}
\thanks{T.~Sepp\"al\"ainen was partially supported by National Science Foundation grants  DMS-1854619, DMS-2152362, and DMS-2448375, by Simons Foundation grant 1019133, and by the Wisconsin Alumni Research Foundation.}

\begin{document}

\title[Strong existence and uniqueness of the Busemann process in CGM]{Strong existence and uniqueness of the tilt-indexed Busemann process in the planar corner growth model}  

\begin{abstract}
We show that the Busemann process indexed by ``tilts" in the super-differential of the limit shape exists and is unique in the strong sense in the i.i.d.\ planar corner growth model. This means that every probability space that supports the field of i.i.d.~weights supports a copy of the process and any two realizations of the process are equal almost surely. 
\end{abstract}

\setcounter{tocdepth}{1}

\maketitle

\tableofcontents 

\section{Introduction} 
The \textit{corner growth model}, or \textit{directed last-passage percolation} (LPP), is a cornerstone model of random growth in the Kardar-Parisi-Zhang universality class which can profitably be thought of as a directed version of first-passage percolation (FPP), a prototypical example of a random (psuedo-)metric on the lattice. 

FPP is constructed by assigning non-negative weights to the edges of the lattice and then defining the distance between sites to be the infimum over the weights collected along all self-avoiding paths between those sites. As usual, the optimizing paths are known as geodesics. In FPP, substantial recent attention \cite{Dam-Han-14, Dam-Han-17, Hof-08, Ahl-Hof-16-} has focused on the structure of semi-infinite geodesics.

The planar corner growth model has a similar structure with a few differences. Edge weights are replaced by vertex weights and the set of admissible paths is restricted to those which either go up or right in each step. This path restriction allows the vertex weights to take both negative and positive values. By convention, we also replace the minimum over paths with a maximum over paths. This model arises naturally in the context of tandem queueing, for example, where the passage time satisfies the same recursion as the time at which labeled customers are served at labeled service stations. Once again, we have optimizing paths, which we call geodesics by way of analogy to first-passage percolation. Similarly, significant recent attention has been focused on the structure of semi-infinite geodesics.

In the context of metric geometry, a frequently used tool to understand the structure of semi-infinite geodesics are objects known as Busemann functions, originally introduced in \cite{Bus-55}. In lattice growth models, a slight generalization of this notion, which we call \textit{generalized Busemann functions}, can be viewed as the natural coupling of all ergodic and translation invariant stationary distributions of the model. See the introductions to \cite{Jan-Ras-Sep-23, Jan-Ras-Sep-23-1F1S-, Gro-Jan-Ras-25} for a discussion of these connections and a more thorough review of the related literature.

The last decade has seen several constructions of generalized Busemann functions in various models. In both first- \cite{Dam-Han-14} and last-passage percolation \cite{Geo-Ras-Sep-16, Mai-Pra-03}, as well as related models like directed polymers \cite{Jan-Ras-20-aop} and generalizations like random walks in random potentials \cite{Gro-Jan-Ras-25-}, most of this work shows what is known as \textit{weak existence} in the stochastic analysis literature. This means that existence statements take the form: ``there exists a probability space on which generalized Busemann functions are defined.''

Throughout the study of random growth models, a quantity known as the \textit{limit shape} or \textit{free energy density} plays a central role. In general, the existence results mentioned above show that for each element of the subdifferential (superdifferential in LPP)  of the shape, a generalized Busemann function field exists. In 
planar last-passage percolation, these can be glued together by monotonicity to form a \textit{Busemann process} indexed by negatives of the elements of the superdifferential, which we call \textit{tilts}. It is known \cite{Jan-Ras-Sep-23} that this Busemann process encodes geometric properties of geodesics into its analytic behavior.

A line of work originating with Newman \cite{New-95} shows that Busemann functions are almost sure attractors in an appropriate sense, which proves both \textit{strong existence} (meaning that every probability space supporting the weights supports the Busemann process) and \textit{strong uniqueness} (any two generalized Busemann functions associated to the same deterministic tilt are equal almost surely). These methods require control of the curvature and differentiability properties of the limit shape. Proving these types of estimates is a long-standing and difficult open problem outside of exactly solvable models. A collection of results in this vein implying strong existence and uniqueness in the exactly solvable Exponential LPP originally appear in \cite{Cat-Pim-13, Cou-11, Fer-Pim-05}.

About a decade ago, Ahlberg and Hoffman \cite{Ahl-Hof-16-} introduced a theory of ``random coalescing geodesics" in FPP. One of the contributions of that paper is a proof of strong existence of the Busemann functions without unproven hypotheses on the limit shape. Namely, what we call generalized Busemann functions previously constructed in the weak sense in planar first-passage percolation by Damron and Hanson in \cite{Dam-Han-14}, can be pulled back to the original probability space. In \cite{Ahl-Hof-16-}, this was phrased in terms of the associated shift-covariant geodesics.  We show below that the geodesic and Busemann function perspectives are equivalent.  They also introduce a labeling scheme that allows them to prove a similar strong uniqueness statement to the one we prove below. There are technical differences between our treatment and theirs, but our methods are largely inspired by those of \cite{Ahl-Hof-16-}. We do not use the geodesic labeling scheme introduced in \cite{Ahl-Hof-16-}, preferring to work directly with the tilt indexing.

As discussed in \cite{Bak-Cat-Kha-14, Bak-Kha-18, Jan-Ras-Sep-23-1F1S-, Jan-Ras-Sep-23} among other places, one can essentially think of fields of Busemann functions in 
stochastic growth models as eternal solutions of a (discrete) stochastic partial differential equation. From that stochastic analytic point of view, whether or not strong solutions exist and if they satisfy strong uniqueness are standard topics of intrinsic interest. We refer the reader to \cite{Kur-07, Kur-14} for a high level perspective on strong existence and uniqueness of stochastic equations. Aside from this intrinsic interest, the purpose of this paper, and of strong existence in particular, is to provide the right setting for future work. Working on an extended probability space is  unnatural and introduces technical issues. For example,  weight modification arguments that  control infinite geodesics generated by Busemann functions become delicate because the extra noise from the extension  is not independent of the weight field. This extended space is also in general only stationary, rather than ergodic, under lattice translations. This complicates the application of standard results. See, for example, Remark A.5 in \cite{Jan-Ras-Sep-23}, which discusses working around this technical issue.

Strong uniqueness is of interest for related reasons. Without control on the regularity of the limit shape, the only uniqueness results that existed concerned the finite-dimensional marginal \cite{Cha-94, Pra-03, Fan-Sep-20} distributions of generalized Busemann functions. We include \cite{Fan-Sep-20}  here because the proof of Theorem 5.6 in that work can be made general under the same hypotheses as \cite{Cha-94}, though it is not phrased as such. These results were not sufficiently strong to show that any two such processes would have to be equal, leaving open the possible existence of multiple Busemann processes. Pre-existing distributional uniqueness results also require extra hypotheses: applying the results of \cite{Cha-94, Fan-Sep-20} would require the weights to be unbounded from above, while applying \cite{Pra-03} would require the weights to be bounded from below.

Our results, like those of \cite{Ahl-Hof-16-}, are limited to the planar setting. The reason why both strong existence and uniqueness are accessible in a planar setting but remain open in non-planar settings is geodesic coalescence and a path ordering coming from the nearest-neighbor paths and planarity.  In both first- and last-passage percolation, this coalescence comes from the Licea-Newman argument \cite{Lic-New-96}.

We rely on a seminal result of Martin \cite{Mar-04} concerning the curvature of the limit shape of the i.i.d.\ corner growth model near the boundary. Presently, this curvature result is restricted to i.i.d.\ weights, because its proof relies on the exact solvability of the so-called Sepp\"al\"ainen-Johansson model with i.i.d.\ Bernoulli weights \cite{Sep-98-aop-2}.  Ahlberg and Hoffman \cite{Ahl-Hof-16-} work somewhat more generally in their more abstract Condition A2. We have made no attempt to generalize beyond the i.i.d.\ setting. 

As a consequence of strong existence, we also obtain a significant extension of the previously-known ergodicity properties of the Busemann process in the i.i.d.\ corner growth model. Prior to this work, it was known that Busemann functions which correspond to extremal vectors in the subdifferential were separately ergodic under the $e_1$ and $e_2$ shifts \cite{Cha-94}. Because the Busemann process can be realized as a function of an i.i.d.\ field, this can now be upgraded to strong mixing under all nontrivial lattice shifts.

\subsection{Outline}
In Section \ref{sec:set}, we state the assumptions of the model, define terms, recall facts from the previous literature. We state our main results in Section \ref{sec:main}. The proofs of the two main results concerning strong existence and uniqueness appear in Section \ref{sec:unique}. Section \ref{sec:equiv} then shows an equivalence between two ways of setting up the problems we consider. Appendix 
\ref{app:aux} collects some technical lemmas.  Appendix \ref{a:geoBus} proves that every nontrivial geodesic generates a finite Busemann function.

\subsection{Notation and conventions}  
 $\Z$ are the integers, $\Z_+$ the nonnegative and $\Z_-$ the nonpositive integers. $\N=\{1,2,3,\dotsc\}$. $\bbQ$ are the rational numbers, and $\R$ are the real numbers.  $\Z_{>a}$ are integers $>a$. Similarly for $\Z_{<a}$, $\Z_{\ge a}$, and $\Z_{\le a}$.
$\lzb a,b\rzb=[a,b]\cap\Z$.   

When a symbol is needed, $\leb$ denotes the Lebesgue measure on $[0,1]$. We write integrals with respect to this measure using the standard notation $\int_0^1 f(s)\,ds$.

The end of a non-proof structured environment ends with a $\triangle$ to help delineate these from the rest of the text.

\section{Setting and past results}\label{sec:set}

\subsection{Probability space}\label{sec:probsp}
$(\Omhat,\kShat,\Phat)$ denotes a generic Polish probability space equipped with a group of continuous bijections $\That=\{\That_x:x\in\Z^2\}$ from $\Omhat$ onto itself.
In particular, $\That_0$ is the identity map and $\That_x\That_y=\That_{x+y}$ for all $x,y\in\Z^2$. 
  $\Ehat$ denotes expectation under $\Phat$. 
A standing assumption is that  
    \[\text{$\Phat$ is invariant under $\That_x$ for each $x\in\Z^2$.}\] 
Let $\cIhat$ be the $\sigma$-algebra of events that are invariant under the group of shifts $\That$. A generic element of $\Omhat$ is denoted by $\what$. As usual, the $\what$ can be dropped from the arguments of random variables.

We are given random variables $\w_x:\Omhat\to\R$, $x\in\Z^2$, which satisfy the shift-covariance property 
\begin{align}
\w_x(\That_z\what)=\w_{x+z}(\what) \label{eq:Thatcov}
\end{align}
almost surely under $\Phat$.  We assume also that 
\begin{align}\label{iid}
\{\w_x:x\in\Z^2\} \text{ are i.i.d.\ under }\Phat\text{ and satisfy } \Ehat[|\w_0|^{2+\epsilon}]<\infty \text{ for some }\epsilon>0.
\end{align}

To dispense with trivialities we assume
\begin{align}\Ehat[\w_0^2]>\Ehat[\w_0]^2.\label{var>0}\end{align}

Let $\kS$ denote the $\sigma$-algebra on $\Omhat$ generated by $\w=\{\w_x:x\in\Z^2\}$. 
For $x\in\Z^2$ let $\kSfor_x\subset\kS$ denote the $\sigma$-algebra on $\Omhat$ generated by 
$\{\w_y:y-x\in\Z^2_+\}$. 

$\Omega$ denotes the product space $\R^{\Z^2}$ and $\sF= \sB(\bbR^{\Z^2})$ its Borel $\sigma$-algebra. We   use the symbol $\w=(\w_x)_{x\tsp\in\tsp\Z^2}$ to denote  both the $\Omhat\to\Omega$ mapping $\what\mapsto\w(\what)=(\w_x(\what))_{x\tsp\in\tsp\Z^2}$ and a generic element of $\Omega$. The shifts $T=\{T_x : x \in \Z^2\}$ on $\Omega$ are defined by 
\begin{align}\label{T-cov}
  (T_z\w)_x=\w_{x+z}.
\end{align}
The  shifts on the two spaces are related via the following identity, valid for $x,z\in\bbZ^2$ and $\Phat$-almost all $\what \in \Omhat$:
\begin{align}\label{w-cov}
(T_z[\w(\what)])_x\overset{\eqref{T-cov}}=[\w(\what)]_{z+x}\overset{\rm(def.)}=\w_{z+x}(\what)\overset{\eqref{eq:Thatcov}}=\w_x(\That_z\what).
\end{align}
 For $x\in\Z^2$ let $\sFfor_x\subset\sF$ denote the $\sigma$-algebra on $\Omega$ generated by $\{\w_y:y-x\in\Z^2_+\}$.
We denote by $\P(\acdot)=\Phat\{\what: \w(\what)\in\acdot\}$ the probability measure induced by pushing forward $\Phat$ by the map $\what \mapsto \w(\what)$.  By \eqref{iid}, $\P$ is an i.i.d.\ product measure on $\Omega$.

\subsection{Path spaces and order relations}
A path $\pi_{m:n}=(\pi_i)_{i=m}^n$ of points in $\Z^2$ is  \emph{up-right} if it only takes steps in $\{e_1,e_2\}$, meaning $\pi_{i+1} - \pi_{i} \in \{e_1, e_2\}$ for all integers $i \in \lzb m, n-1\rzb$.  This definition extends to semi-infinite and bi-infinite up-right paths.
 Paths are indexed by antidiagonal levels, that is, 
$\pi_i\cdot(\evec_1+\evec_2)=i$ for all $i$ in the relevant range, with one exception:  
the symbol 
$o$ denotes the index of the point of origin of a path. To clarify, if $u$ is the first point of $\gamma$ and
$u\cdot(\evec_1+\evec_2)=m$, then   $\gamma_o=\gamma_m=u$.  
When $u\le v$ are points on a path $\gamma$,   $\gamma_{u:v}$ is the segment of $\gamma$ from $u$ to $v$, including the endpoints $u$ and $v$. Then $\gamma_{\gamma_m:\gamma_n}$ is abbreviated by $\gamma_{m:n}$. If $m$ is an index and $u\in\gamma$, then mixtures $\gamma_{m:u}$ and $\gamma_{u:m}$ are also entirely unambiguous. 

For $\ell\in\Z$ and $u\in\Z^2$ with $u\cdot(e_1+e_2)=\ell$, $\pathsp_u$ denotes the space of semi-infinite up-right paths $\gamma=\gamma_{\ell:\infty}$ on $\Z^2$ that start at $\gamma_\ell=u$. This path space is compact in the product-discrete topology, which can be metrized by  $d(\gamma, \pi) = \sum_{i=\ell}^\infty 2^{-(i-\ell+1)}\one_{\{\gamma_i \neq \pi_i\}}$.   $\pathsp=\bigcup_{u\tsp\in\tsp\Z^2}\pathsp_u$ is the space of all up-right lattice paths.

We use $\preceq$ to denote southeast type partial order relations on various spaces. Relation  $a\precneq b$ means $a\preceq b$ but $a\ne b$. 

For $h,h'\in\R^2$,  $h\preceq h'$ means $h\cdot e_1\le h'\cdot e_1$ and $h\cdot e_2\ge h'\cdot e_2$.   In particular, the simplex $[e_2, e_1]=\{(t,1-t)\in\R^2: 0\le t\le1\}$ of direction vectors is ordered so that $\zeta\preceq\eta$ if and only if  $\zeta\cdot e_1\le \eta\cdot e_1$. 

Between finite and infinite paths the relations $\preceq$ and $\precneq$ are interpreted coordinatewise on the common part of the domains of these paths.   That is, if $\lzb k,\ell\rzb=\lzb m,n\rzb\cap\lzb m',n'\rzb$, then   $\pi_{m:n}\preceq \pi'_{m':n'}$ means that  $\pi_i\preceq \pi'_i$ for each $i\in\lzb k,\ell\rzb$, while $\pi_{k:\ell}\precneq \pi'_{m:n}$ means that $\pi_i\preceq \pi'_i$ for each $i\in\lzb k,\ell\rzb$ and $\pi_i\precneq \pi'_i$ for at least one index $i\in\lzb k,\ell\rzb$.

An additional asymptotic version of southeast ordering  is defined for semi-infinite up-right paths $\pi,\gamma\in\pathsp$ as follows: $\gamma_{k:\infty}\apreceq \pi_{\ell:\infty}$ means that $\pi_{\ell:\infty}$ is \emph{eventually} weakly to the right of  $\gamma_{k:\infty}$.  
Precisely,  there exists an integer $m\ge k\vee\ell$ such that $\gamma_n\preceq \pi_n$ for all integers $n\ge m$.  
If there exists an integer $m\ge k\vee\ell$ such that $\gamma_n=\pi_n$ for all $n\ge m$, then these paths {\it coalesce}, abbreviated by $\gamma_{k:\infty}\coal \pi_{\ell:\infty}$. $\gamma_{k:\infty}\coal \pi_{\ell:\infty}$ is equivalent to  $\gamma_{k:\infty}\apreceq \pi_{\ell:\infty}$ and $\gamma_{k:\infty}\apreceq \pi_{\ell:\infty}$.

For $B,B'\in\R^{\Z^2}$, $B\preceq B'$ means $B(x,x+e_1)\ge B'(x,x+e_1)$ and $B(x,x+e_2)\le B'(x,x+e_2)$ for all $x\in\Z^2$.  This order relation is relevant for functions called cocycles introduced below in Section \ref{s:cocycle}. 

These partial orders turn out to be total orders   and  consistent with each other on the spaces we study:  the super-differential of the shape function, semi-infinite geodesic trees, and recovering cocycles. 

\subsection{Last-passage percolation}
Given weights $\w=(\w_x)_{x\in\Z^2}\in\R^{\Z^2}$ and two distinct points $u\le v$ (coordinatewise) in $\Z^2$, let
\[\Lpp_{u,v}=\max\Bigl\{\sum_{x\in\pi\setminus\{v\}}\w_x:\pi\text{ is up-right from $u$ to $v$}\Bigr\}.\]
A maximizing path in the above is called a \emph{point-to-point geodesic} (from $u$ to $v$). The \emph{rightmost geodesic} $\pi$ from $u$ to $v$ is the unique geodesic between the two points that is to the right of any other geodesic from $u$ to $v$: if $\gamma$ is another geodesic, then $\gamma\preceq\pi$.

A \emph{semi-infinite geodesic}, starting at $u\in\Z^2$ with $m=u\cdot(e_1+e_2)$, is a path $\pi$ with $\pi_m=u$, $\pi_{i+1}-\pi_{i}\in\{e_1,e_2\}$ for all integers $i\ge m$, and such that $\pi_{k:\ell}$ is a geodesic from $\pi_k$ to $\pi_\ell$, for any pair of integers $\ell>k\ge m$. A semi-infinite geodesic is said to be \emph{locally-rightmost} if each finite segment is the rightmost geodesic between its endpoints. 

When $\pi$ and $\gamma$ are both locally-rightmost geodesics that start at the same point $\pi_o=\gamma_o$, $\pi\apreceq\gamma$ is equivalent to $\pi\preceq\gamma$.

For $u\in\Z^2$ and $m\in\Z$ with $m\ge k=u\cdot(e_1+e_2)$ let $\Geo_{u,m}^\w$ denote the set of rightmost geodesics  
 that start at $u$ and end at some $x\in u+\Z^2_+$ at level 
$m=x\cdot(\evec_1+\evec_2)\ge k$. 
The set of semi-infinite up-right paths $\pi_{\parng{k}{\infty}}$ that start at $\pi_k=u$ and satisfy $\pi_{k:m}\in \Geo^\w_{u,m}$  for all $m\in\Z_{\ge k}$ is  denoted by $\Geo_u^\w$. This is the set of all locally-rightmost semi-infinite geodesics started at $u$. A discussion of measurability of $\Geo_u^\w$ appears in Appendix \ref{app:aux}. It is immediate from the definition that 
\be\label{Geo55}u+\Geo^{T_u\w}_0=\Geo_u^\w\quad\text{for all }u\in\Z^2\text{ and }\w\in\Omega. \ee 

The (deterministic) uniqueness of finite rightmost geodesics implies that $\Geo_u^\w$ is a tree. There are two \textit{trivial} locally-rightmost semi-infinite geodesics in $\Geo_u^\w$ given by $u + \Z_+e_1$ and $u + \Z_+e_2$. This follows from the path structure. We say that a semi-infinite geodesic is non-trivial if it is not one of these trivial semi-infinite geodesics. 

The uniqueness of rightmost point-to-point geodesics implies that $\preceq$ is a total order on $\Geo_u^\w$. Precisely, the following three facts hold for each pair of geodesics $\pi$ and $\gamma$  in $\Geo_u^\w$: 
\begin{align}
  \label{G0.i}  &\text{$\pi\preceq \gamma \preceq \pi$ is equivalent to  $\pi= \gamma$.}\\[3pt]
   \label{G0.ii}  &\text{If $\pi\precneq \gamma$ then  $\pi$ and $\gamma$ separate at some point and never intersect again.}  \\[3pt] 
    \label{G0.iii}  &\text{Exactly one of $\pi\precneq \gamma$, $\pi\succneq \gamma$ and  $\pi=\gamma$ holds.}
\end{align}

\subsection{Limit shape}
By the shape theorem in \cite{Mar-04}, there exists a \textit{shape function} $\shape:\bbR_+^2 \to \bbR$ such that with $\P$-probability one
\be\label{sh-th} 
\lim_{n\to\infty} \;\max_{x\, \in \,\bbZ^2_+ : \, |x|_1 = n} \frac{\abs{\Lpp_{0,x} - \shape(x)}}{n} = 0.
\ee
$\shape$  is symmetric, concave, and positively homogeneous of degree one. Homogeneity implies $\shape$ is determined by its restriction to $\Uset = [e_2,e_1]$. 

The super-differential of $\shape$ at $\xi\in \R_+^2$ is
\be
\partial \shape(\xi) = \{h \in \R^2 : \shape(\zeta)-\shape(\xi) \leq h\cdot (\zeta-\xi) \text{ for all }\zeta\in\R_+^2\}. 
\ee
By homogeneity, $\partial \shape(\xi)=\partial \shape(c\xi)$ for any $c>0$. Thus $\partial \shape(\bbullet)$ is also determined by points on $\Uset$. Concavity implies the existence of one-sided derivatives at relative interior points $\xi\in\ri\Uset$:
\begin{align*}
\nabla \shape(\xi \pm) \cdot e_1 = \lim_{\e \searrow 0} \frac{\shape(\xi \pm \e e_1) - \shape(\xi)}{\pm \e} \quad{ and }\quad \nabla \shape(\xi \pm) \cdot e_2 = \lim_{\e \searrow 0} \frac{\shape(\xi \mp \e e_2) - \shape(\xi)}{\mp \e}.
\end{align*}
By \cite[Lemma 4.7(c)]{Jan-Ras-18-arxiv}  differentiability of $\shape$ at $\xi\in\ri\Uset$ is the same as $\nabla \shape(\xi+) = \nabla \shape(\xi -)$. More generally, these values are the extreme points of the convex set $\partial \shape(\xi)$. 

In addition to the shape theorem, Martin proved universal asymptotics of the limit shape which play a key role in many of our arguments. By 
Theorem 2.4 in \cite{Mar-04}, with  $\mu = \bbE[\w_{0}]$ and $\sigma = \sqrt{\Var(\w_{0})}$, 
\be\label{eq:Martin} 
\shape(1,t)=\shape(t,1) = \mu + 2\sigma \sqrt{t} + o\bigl(\sqrt{t}\tspb\bigr) \quad\text{ as $t \searrow 0$}.
\ee 
This  implies that the limit shape has infinitely many faces and  $\partial \shape(e_1) = \partial \shape(e_2) = \emptyset$. These facts are often critically important in nontriviality arguments and this is the primary reason why we must assume the weights are i.i.d.\ in \eqref{iid}.

\begin{remark}
The results of \cite{Mar-04} mentioned above have a slightly weaker moment assumption than in \eqref{iid}. The existence result in \cite{Jan-Ras-20-aop} recorded below as Theorem \ref{thm:exist} also relies on a variational characterization of the limit shape from \cite{Geo-Ras-Sep-16}, which was proven under our stronger moment hypothesis.
\end{remark}

An important index set for Busemann functions is the total superdifferential of the shape function, denoted by 
\begin{align}
\partial \shape(\Uset) = \{h \in \R^2 : \text{there exists } \xi \in \Uset \text{ with }h \in \partial \shape(\xi)\}.\label{eq:pargUset}
\end{align}
In the sequel, we call a subset of negatives of elements of this set \textit{tilts}. By Lemma 4.6(a) in \cite{Jan-Ras-20-aop}, $\preceq$ is a total order on $\partial\shape(\Uset)$.
Combining equation (4.14) in \cite{Geo-Ras-Sep-17-ptrf-1} with Lemma 4.6(c) in \cite{Jan-Ras-20-aop}, one sees that $\partial \shape(\Uset)$ is a one dimensional curve. The connection to the shape is illustrated in Figure \ref{fig:shape}. Lemma 4.6(c) implies in particular that $\partial\shape(\Uset)$ is the union of the line segments $[\nabla\shape(\xi+),\nabla\shape(\xi-)]$ over $\xi\in\ri\Uset$. 

\begin{figure}
\includegraphics[scale=.5]{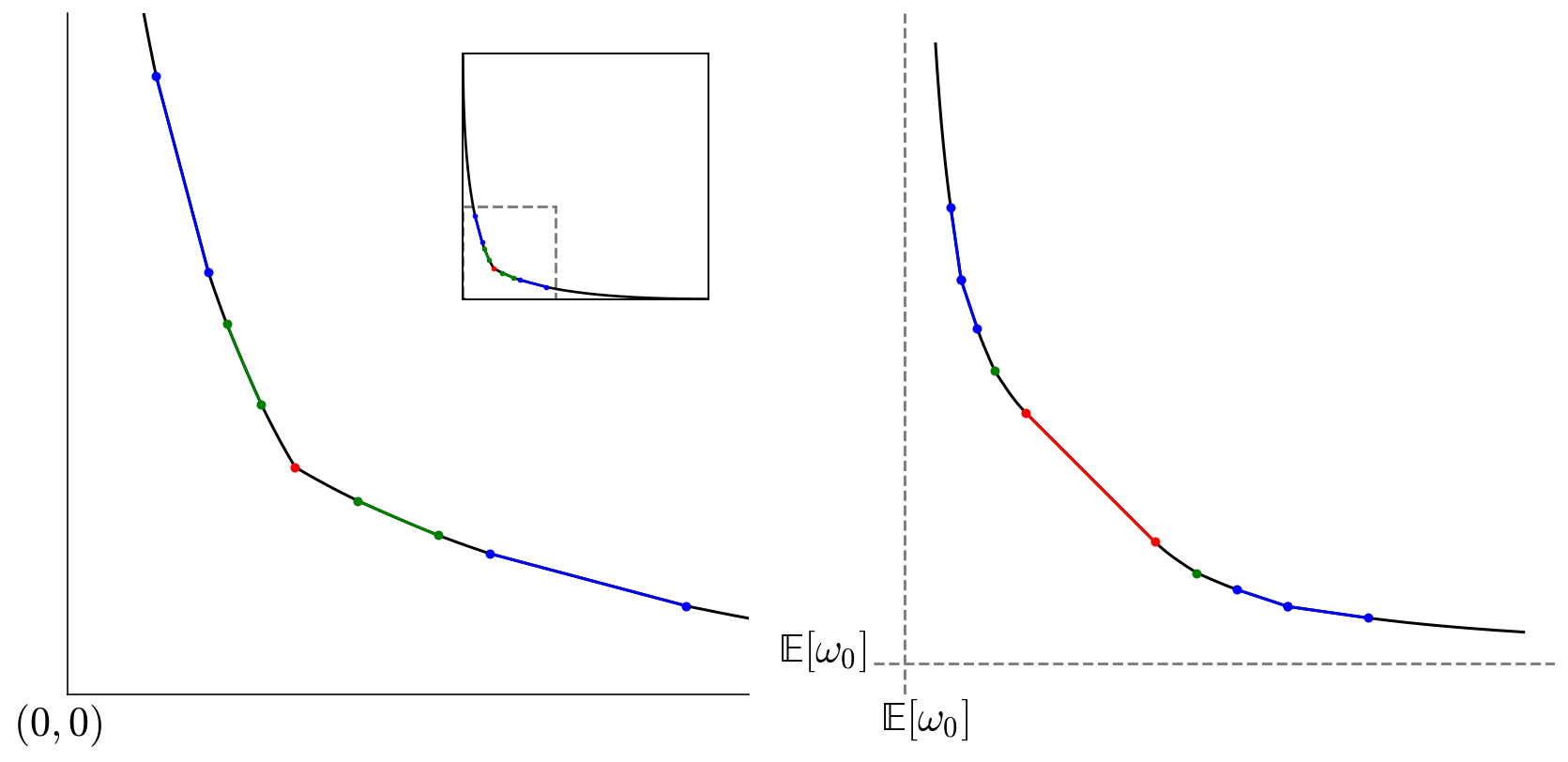}
\caption{\small A level set $\{x\in\R_+^2:\shape(x)=1\}$ of a limit shape (left) and its associated 
super-differential curve $\partial \shape(\Uset)$ (right). The full level set of the limit shape is depicted in the inset image on the left and the dashed lines indicate the enlarged portion. 
This shape has four linear segments and a cusp on the diagonal. 
It is differentiable at the endpoints of the two inner segments (green) 
and non-differentiable at five points: the four endpoints of the outer segments (blue) and the cusp on the diagonal (red). 
Non-differentiability points of the shape correspond to line segments in the super-differential, which represent the intervals of slopes of supporting lines at those points, while linear segments of the shape correspond to non-differentiability points in the super-differential.}\label{fig:shape}
\end{figure}

\subsection{Recovering cocycles and Busemann functions} 
\label{s:cocycle} 
\begin{definition}
A function $\Busgeo:\Z^2\times\Z^2\to\R$ is a \emph{cocycle} if 
\be\begin{aligned}\label{cocycle}
\Busgeo(x,y)+\Busgeo(y,z)=\Busgeo(y,z)\quad\text{for all }x,y,z\in\Z^2.
\end{aligned}\qedhere\ee
\end{definition}

\begin{definition}
Given real weights $\w=(\w_x)_{x\in\Z^2}\in\R^{\Z^2}$, a function $\Busgeo:\Z^2\to\R$ is said to \emph{recover} the weights $\w$ if it satisfies the following \emph{recovery property}:
\be\begin{aligned}\label{recovery}
\Busgeo(x,x+e_1)\wedge\Busgeo(x,x+e_2)=\w_x\quad\text{for all $x\in\Z^2$}.
\end{aligned}\qedhere\ee
\end{definition}

Given a recovering cocycle $\Busgeo$, a semi-infinite up-right path $\pi$ is called an $\Busgeo$-geodesic in weights $\w$ if it satisfies 
    \begin{align}\label{By rule}
\Busgeo(\pi_i,\pi_{i+1}) =\w_{\pi_i}   
    \end{align}
at each step of $\pi$.  
Such a path is always a geodesic in weights $\w$ because for any other up-right path $(x_i)_{i=m}^n$ from $u=\pi_m$ to $v=\pi_n$, 
\[\sum_{i=m}^{n-1}\w_{x_i}
\le\sum_{i=m}^{n-1}\Busgeo(x_i,x_{i+1})
=\Busgeo(u,v)=\sum_{i=m}^{n-1}\Busgeo(\pi_i,\pi_{i+1})=\sum_{i=m}^{n-1}\w_{x_{\pi_i}}.
\]

By Theorem \ref{thm:g400} in Appendix \ref{a:geoBus}, under \eqref{iid}, there exists an event $\Omega_0\in\kS$ of full $\P$-probability on which, for every nontrivial semi-infinite geodesic $\pi$, the limits 
\be\label{B-gen4} 
\Busgeo_\pi(\w,x,y)=\lim_{n\to\infty}(\Lpp_{x,\pi_n}(\w)-\Lpp_{y,\pi_n}(\w))\ee
define a recovering  cocycle. This is the definition originally introduced by Busemann \cite{Bus-55} in metric geometry. Thus, $\Busgeo_\pi$ is called the \emph{Busemann function} generated by $\pi$.  
The definition of $\Busgeo_\pi$ implies  \eqref{By rule} and thereby $\pi$ itself is  an  $\Busgeo_\pi$-geodesic.

The fact that $\Lpp_{x+z,y+z}(\w)=\Lpp_{x,y}(T_z\w)$ gives that if $\pi$ is a semi-infinite geodesic in the weights $T_z\w$ then $z+\pi=(z+u:u\in\pi)$ is a
semi-infinite geodesic in the weights $\w$ and
	\begin{align}\label{B-cov}
	\Busgeo_{z+\pi}(\w,x+z,y+z)=\Busgeo_\pi(T_z\w,x,y).
	\end{align}

Lemma \ref{lm:B-order} states that, $\P$-almost surely,
for any nontrivial $\gamma\apreceq \pi$ in $\Geo^\w$, $\Busgeo_\gamma\preceq \Busgeo_\pi$. 
    Consequently, if $\gamma\coal \pi$, then $\Busgeo_\gamma=\Busgeo_\pi$.

\subsection{Generalized Busemann functions, the Busemann process, and existence}

\begin{definition}
    A measurable function $\Bhat:\Omhat\times\Z^2\times\Z^2\to\R$ is 
\emph{shift-covariant} if for $\Phat$-almost every $\what$,
\begin{align}\label{covariance}
\Bhat(\what,x+z,y+z)=\Bhat(\That_z\what,x,y)\quad\text{for all $x,y,z\in\Z^2$}.
\end{align}
$\Bhat$ is $L^1(\Omhat,\kShat,\Phat)$ if
\be\begin{aligned}\label{BL1}
\Ehat[\abs{\Bhat(x,y)}]<\infty\quad\text{for all $x,y\in\Z^2$}. 
\end{aligned}\ee
When the underlying probability space is clear from context, we simply say that $\Bhat$ is $L^1$.
\end{definition}

The main objects that we consider in this work are shift-covariant, recovering, $L^1$ cocycles, which we call generalized Busemann functions.

\begin{definition}
A \emph{shift-covariant recovering $L^1(\Omhat,\kShat,\Phat)$ cocycle} is a shift-covariant $L^1$ measurable function 
$\Bhat:\Omhat\times\Z^2\times\Z^2\to\R$
such that $\Bhat$ is $\Phat$-almost surely a recovering cocycle.  
The collection of all such cocycles is denoted by $\cKhat$, and its elements are called \emph{generalized Busemann functions}.
\end{definition}

For $\Bhat\in\cKhat$ 
define the random 2-vector $\hhB(\Bhat)=\hhB(\Bhat,\what)\in\R^2$ via
\begin{align}\label{h-def}
	\hhB(\Bhat)\cdot e_i=-\Ehat[\Bhat(0,e_i)\,|\,\cIhat],\quad i\in\{1,2\}.
\end{align}
By \cite[Theorem 4.4]{Jan-Ras-20-aop} (see \cite[Theorem B.3]{Jan-Ras-18-arxiv} for the details),  for $\Phat$-a.e.\ $\what$, 
\begin{align}\label{B-shape}
\lim_{n\to\infty}n^{-1}\max_{\abs{x}_1\le n}\abs{\Bhat(\what, 0,x)+\hhB(\Bhat, \what)\cdot x}=0.
\end{align}
We have the following lemma connecting generalized Busemann functions to the superdifferential of the shape function. Recall the set $\partial \shape(\Uset)$  defined in  \eqref{eq:pargUset}.

\begin{lemma}\textup{\cite[Lemma 4.5]{Jan-Ras-20-aop}}\label{lem:hsupdif}
A generalized Busemann function $\Bhat \in \cKhat$ has the following properties: 
\begin{enumerate}[label={\rm(\alph*)}, ref={\rm\alph*}]   \itemsep=3pt  
\item\label{hsupdif.a} $-\hhB(\Bhat)$ takes values in $\partial \shape(\Uset)$, $\Phat$-almost surely.
\item\label{hsupdif.b} If $-\Ehat[\hhB(\Bhat)]\in\partial\shape(\xi)$ for some $\xi \in \Uset$, then $-\hhB(\Bhat)\in\partial\shape(\xi)$ $\Phat$-almost surely.
\item\label{hsupdif.c} If $-\Ehat[\hhB(\Bhat)] \in\{\nabla \shape(\xi+),\nabla\shape(\xi-)\}$ for some $\xi \in \Uset$, then $\hhB(\Bhat) = \Ehat[\hhB(\Bhat)]$ $\Phat$-almost surely. \qedhere
\end{enumerate}
\end{lemma}

As mentioned in the introduction, when the weights are unbounded from above \cite{Cha-94} or bounded from below \cite{Pra-03}, it is known that for each $h \in -\partial\shape(\Uset)$ there exists at most one distribution of a generalized Busemann function satisfying $\Phat\{\what:\hhB(\Bhat,\what)=h\}=1$. In the stochastic analysis literature, such a result is referred to as weak uniqueness \cite{Kur-07,Kur-14}. However, weak uniqueness falls short of establishing the strong existence that is one of the central goals of this work.


We begin our discussion by recalling that weak existence, meaning existence of a probability space which supports these cocycles, is already known for cocycles with a mean vector given by an extreme point of the superdifferential. By an extension theorem argument, the problem of showing existence of shift-covariant recovering cocycles on an extended probability space is equivalent to constructing translation invariant stationary distributions for LPP (or, equivalently, for G/G/1/$\infty$ queueing models). See Section 2.4 of \cite{Jan-Ras-20-jsp} for a sketch of how such an argument goes.

A partial result proving existence of certain stationary queueing fixed points for the tandem queueing model connected to the general i.i.d.\ weight corner growth model was originally established by Mairesse and Prabhakar \cite{Mai-Pra-03} under the assumption that the weights are bounded from below with $>2$ moments, but phrased in queueing language. These were used to generate generalized Busemann functions  for the corner growth model in \cite{Geo-Ras-Sep-17-ptrf-1}. Connections to geodesics were explored in \cite{Geo-Ras-Sep-17-ptrf-2}. \cite{Jan-Ras-20-aop} subsequently removed the boundedness below requirement for existence.  By monotonicity, these constructions also build a \textit{Busemann process} on the extended space. This is a shift-covariant, recovering cocycle-valued stochastic process indexed by $-\partial \shape(\Uset)\times\{+,-\}$ as described by the next theorem.

\begin{theorem}{\rm\cite[Theorem 4.7]{Jan-Ras-20-aop}}\label{thm:exist} There exists a probability space $(\Omhat,\kShat,\Phat)$,  equipped with an additive group of continuous bijections $\That=\{\That_x:x\in\Z^2\}$ and satisfying the hypotheses of Section \ref{sec:probsp}, on which there exists a stochastic process
\begin{align} \label{Bhat65}   
    \bigl(\Bhat^{h\sigg}(x,y) : x,y\in\bbZ^2, h \in -\partial \shape(\Uset),\, \sigg \in \{+,-\}\bigr)
\end{align}
with the following properties:
\begin{enumerate}  [label={\rm(\alph*)}, ref={\rm\alph*}]   \itemsep=3pt
\item \textup{(}No $\pm$ distinction at fixed $h$\textup{)} For each $h \in -\partial \shape(\Uset)$, 
\[
\Phat\bigl\{\Bhat^{h-}(x,y) = \Bhat^{h+}(x,y)\bigr\}=1.
\]
When $\Bhat^{h-}(\what,x,y) = \Bhat^{h+}(\what,x,y)$, call the common value $\Bhat^h(\what, x,y)$.
\item\label{exist.b} \textup{(}Generalized Busemann function\textup{)} For each $h \in -\partial \shape(\Uset), \Bhat^h \in \cKhat$.  
\item \textup{(}Mean $-h$\textup{)} For each $h \in -\partial \shape(\Uset),$ $\Ehat[\Bhat^h(0,e_i)] = -h \cdot e_i.$
\item \textup{(}Monotonicity\textup{)} For $h,h'\in -\partial \shape(\Uset)$ with $h \cdot e_1 \leq h' \cdot e_1$, all $x \in \bbZ^2$, and $\Phat$-almost every $\what$
\[ 
    \Bhat^{h-}(x,x+e_1) \geq  \Bhat^{h+}(x,x+e_1) \geq  \Bhat^{h'-}(x,x+e_1) \geq  \Bhat^{h'-}(x,x+e_1)  
\]
and
\[ 
    \Bhat^{h-}(x,x+e_2) \leq  \Bhat^{h+}(x,x+e_2) \leq  \Bhat^{h'-}(x,x+e_2) \leq  \Bhat^{h'-}(x,x+e_2).
\]
\item \textup{(}Left-/right-continuity\textup{)} For $\Phat$ almost all $\what$, for all $h\in -\partial \shape(\Uset)$,
\begin{equation*}\begin{aligned}
\Bhat^{h-}(x,y) = \lim_{\substack{-\partial \shape (\Uset)\tsp\ni\tsp  h' \to h\\ h'\cdot e_1 \nearrow h \cdot e_1}} \Bhat^{h\pm}(x,y) \qquad \text{ and } \Bhat^{h-}(x,y) = \lim_{\substack{-\partial \shape (\Uset)\tsp\ni\tsp  h' \to h\\ h'\cdot e_1 \searrow h \cdot e_1}} \Bhat^{h\pm}(x,y).
\end{aligned}\end{equation*}
\item\label{exist.backward} \textup{(}Backward independence\textup{)} For any $I \subset \bbZ^2$, the random variables $\{\w_x, \Bhat^{h\pm}(x,y) : y \geq x,\, x \in I,\, h \in -\partial \shape(\Uset)\}$ are independent of the  weights  $\{\w_x : \forall z\in I , \, x  \not\geq z \}$ behind $I$.\qedhere
\end{enumerate}
\end{theorem}
\begin{remark}\label{rem:compat}
    With the conventions of this paper, time progresses in the south-west direction on the lattice. Condition \eqref{exist.backward} above is then the type of temporal compatibility condition that one often sees when studying weak solutions of stochastic equations. See, for example, the discussion in \cite{Kur-14}. 
\end{remark}

\begin{remark}\label{rem:dir} The Busemann process is also often indexed by directions. This indexing corresponds to restricting the process to the subset of $\partial \shape(\Uset)$ given by the extreme points of each super-differential interval, $-h \in \{\nabla \shape(\xi\sigg) : \xi \in \Uset, \sigg \in \{+,-\}\}$. In principle, it is possible that tilt-indexing gives a richer process if there exist directions of non-differentiability. If the shape is differentiable, as is widely believed to be true in the setting of this work, then the two are equivalent.
\end{remark}

Having now addressed weak existence through Theorem \ref{thm:exist} (so we are not studying the empty set), we return to the study of generalized Busemann functions on a general extended probability space satisfying the assumptions of Section \ref{sec:probsp}. 

\subsection{Cocycle geodesics}
One of the main reasons for geometric interest in generalized Busemann functions is the fact, recorded above as \eqref{By rule}, that any recovering cocycle determines a local rule that constructs semi-infinite geodesics. It is possible in principle for a recovering cocycle to have ties (meaning locations $\Bhat(x,x+e_1)=\Bhat(x,x+e_2)=\w_x)$ even if the weights have a continuous distribution. In this paper, we generate the geodesics that we study with the tie-breaking rule that favors the $e_1$ step. 

For $\Bhat\in\cKhat$ and $x\in\Z^2$ let $\edgehat_x^\Bhat:\Omhat\to\{e_1,e_2\}$ be the $\kShat$-measurable random variable defined by 
	\be\edgehat_x^\Bhat(\what)=\begin{cases}e_1&\text{if }\Bhat(\what,x,x+e_1)\le \Bhat(\what,x,x+e_2),\\ e_2&\text{if }\Bhat(\what,x,x+e_1)> \Bhat(\what,x,x+e_2).\end{cases}\label{eq:e1tie}\ee
Think of $\edgehat^\Bhat=\{\edgehat_x^\Bhat:x\in\Z^2\}$ as arrows at the lattice sites so that $x$ points to $x+\edgehat_x^\Bhat(\what)$.

For $u\in\Z^2$ let $\cgeod{}^{\Bhat,u}(\what)$ denote the path that starts at $u$ and follows the arrows given by $\edgehat^\Bhat(\what)$. 
For $\Phat$-almost every $\what$, these paths satisfy \eqref{By rule} for $\Busgeo=\Bhat(\what)$. Therefore, 
these are semi-infinite geodesics in the environment  $\w(\what)$.  
The $e_1$ tie-breaking rule in \eqref{eq:e1tie} makes $\cgeod{}^{\Bhat,u}(\what)$ the rightmost among all $\Bhat$-geodesics out of $u$. In particular, it is locally rightmost and so $\cgeod{}^{\Bhat,u}(\what)\in\Geo^{\w(\what)}_u$.
$\edgehat^\Bhat$ is a measurable way to encode all these (locally rightmost) semi-infinite $\Bhat$-geodesics. 

The shift-covariance of $\Bhat$ gives, $\Phat$-almost surely, $\edgehat_x^\Bhat(\That_z\what)=\edgehat_{x+z}^\Bhat(\what)$ for all $x,z\in\Z^2$, and hence
\begin{align}\label{xhat-cov}
\cgeod{}^{\Bhat,u+z}(\what)=z+\cgeod{}^{\Bhat,u}(\That_z\what).
\end{align}

Lemma \ref{lem:xhat nontrivial} in the appendix says that these are nontrivial semi-infinite geodesics:
\be\label{no e2}  \Phat\bigl\{ \what:   \cgeod{}^{\Bhat,u}(\what)\not\in\{u+\Z_+e_1,u+\Z_+e_2\}\bigr\}=1.  
\ee


\begin{definition}
Given a shift-covariant recovering $L^1$ cocycle $\Bhat$ we say that it \emph{has coalescing {\rm(}rightmost{\rm)} semi-infinite geodesics} if $\cgeod{}^{\Bhat,u}(\what)\coal\cgeod{}^{\Bhat,v}(\what)$ for all $u,v\in\Z^2$ and $\Phat$-almost all $\what$.  Let $\cKhatcoal\subset\cKhat$ denote the subspace of  shift-covariant recovering $L^1$ cocycles with coalescing {\rm(}rightmost{\rm)} semi-infinite geodesics. 
\end{definition}

Lemma \ref{phih-unique} below states that if $\Bhat\in\cKhatcoal$, then coalescence of its rightmost geodesics implies uniqueness of its locally rightmost geodesics: $\Phat$-almost surely, for all $u\in\Z^2$, $\cgeod{}^{\Bhat,u}(\what)$ is the \emph{unique} $\Bhat$-geodesic in the tree $\Geo^{\w(\what)}_u$.

We next record that the cocycles in Theorem \ref{thm:exist} are examples of members of $\cKhatcoal$.

\begin{theorem}\label{DLR-coal} {\rm\cite[Theorem A.1]{Geo-Ras-Sep-14}}
    In the setting of Theorem \ref{thm:exist}, $\Bhat^h\in\cKhatcoal$ for each $h\in-\partial\shape(\Uset)$.
\end{theorem}

The above result comes from an adaptation of the Licea-Newman \cite{Lic-New-96} coalescence argument. As noted above, the details are provided in
     Theorem A.1 in \cite{Geo-Ras-Sep-14}. 
The backward independence in Theorem \ref{thm:exist}\eqref{exist.backward} implies the finite energy condition used in that coalescence proof.

We close this section with the introduction of analogous notation on  the canonical space $(\Omega,\sF,\bbP)$.

\begin{definition}
Let $\cK$ be the space of shift-covariant recovering $L^1(\Omega,\sF,\P)$ cocycles. 
Let $\cKcoal\subset\cK$ be the subspace of cocycles that have coalescing {\rm(}rightmost{\rm)} geodesics. A cocycle $\Bus\in\cK$ is   
\emph{forward measurable} if, for all $u\in\Z^2$ and  $\{x,y\}\subset u+\Z^2_+$, $\w\mapsto \Bus(\w,x,y)$ is $\sFfor_u$-measurable. We denote by $\cKfor\subset\cK$ the subspace of forward measurable cocycles   
and $\cKcoalfor=\cKcoal\cap\cKfor$. 
\end{definition}
\begin{remark}
In our setting, forward measurability is precisely the usual notion of \textit{adaptedness} to the filtration of the driving noise.
\end{remark}

\subsection{Covariant coalescing systems of random geodesics}
In \cite{Ahl-Hof-16-}, the authors approach closely related problems to those we consider from a different starting point. Instead of focusing on the generalized Busemann functions and the cocycles they generate, Ahlberg and Hoffman begin their study with covariant coalescing systems of geodesics and use these to construct Busemann functions. We briefly define the relevant terms in this section. One of the results in the next section, Theorem \ref{thm:xB}, essentially says that these two approaches are equivalent.

\begin{definition}
A \emph{random (locally rightmost) geodesic} out of $u\in\Z^2$ is a measurable mapping $\pihat:\Omhat\to\pathsp_u$ such that 
\[\Phat\{\what: \pihat(\what)\in\Geo_u^{\w(\what)}\}=1. \qedhere\] 
\end{definition}

\begin{definition}
A \emph{system of random geodesics} is a family of random geodesics 
$\{\pihat^u(\what):u\in\Z^2\}$ such that for each $u\in\Z^2$, $\pihat^u$ is a random geodesic out of $u$.

The system is \emph{coalescing} if 
    \[\Phat\bigl\{\forall u,v\in\Z^2:  \pihat^u(\what)\coal\pihat^v(\what)\bigr\}=1.\]

The system is \emph{shift-covariant} if, $\Phat$-almost surely,
\[\pihat^u(\what)=u+\pihat^0(\That_u\what)\quad \text{ for }  u\in\Z^2. \]

$\Gcchat$ denotes the set of shift-covariant coalescing systems of random geodesics. \qedhere
\end{definition}

\begin{remark}
\cite{Ahl-Hof-16-} refers to what we call a \textit{shift-covariant system of coalescing geodesics} as \textit{random coalescing geodesics}. We use this slightly different terminology because in last-passage percolation, it has been proven that there exist random systems of coalescing geodesics which are not shift-covariant. For example, the system of rightmost geodesics in the exponential last-passage percolation going in the direction of the competition interface rooted at the origin. See \cite[Theorem 3.11]{Jan-Ras-Sep-23}. A similar statement can be expected to hold in first-passage percolation as well.
\end{remark}

Any shift-covariant system of random geodesics can be generated by the member emanating from $0$.
Conversely, every random geodesic $\pihat$ out of $0$ generates a shift-covariant system of random geodesics $\pihat^u$, $u\in\Z^2$, defined by 
\be\label{pihat56}  \pihat^u(\what)=u+\pihat(\That_u\what)\in\Geo_u^{\w(\what)}\quad \text{ for }  u\in\Z^2.\ee
With this notation, $\pihat^0=\pihat$. We will abbreviate $\{\pihat^u:u\in\Z^2\}$ by writing $\pihat^\aabullet$.

We say that a shift-covariant system $\pihat^\aabullet$ of geodesics is \emph{non-crossing} if
    \[\Phat\bigl\{\forall u,v\in\Z^2:  \pihat^u(\what)\coal\pihat^v(\what)\text{ or }\pihat^u(\what)\cap\pihat^v(\what)=\varnothing\bigr\}=1.\]
In plain language, non-crossing geodesics coalesce if they ever touch.   
The recovering cocycle $\Busgeo_{\pihat^u}$ generated by the individual random geodesic $\what\mapsto\pihat^u(\what)$ is the function on $\Omhat\times\Z^2\times\Z^2$ defined in terms of \eqref{B-gen4} by
\be\label{Bpi89}  \Busgeo_{\pihat^u}(\what, x,y)= \Busgeo_{\pihat^u(\what)}(\w(\what), x,y).   
\ee

As one might expect, if one starts with a covariant system of coalescing geodesics, these geodesics generate a single covariant recovering cocycle.
\begin{lemma}\label{Api-cov}
    If $\pihat^\aabullet\in\Gcchat$, then $\Busgeo_{\pihat^0}$ is shift-covariant.
\end{lemma}
The proof of this lemma is almost immediate, but we defer it to the beginning of Section \ref{sec:equiv}. With these basic facts and definitions in place, we now turn to the statements of our results.

\section{Main results}\label{sec:main}

We begin our discussion with the main results of the paper: strong existence and uniqueness of generalized Busemann functions and the Busemann process. We then turn to the connections between our approach and that of \cite{Ahl-Hof-16-}. 

\subsection{Strong existence and uniqueness}
Our first main result is strong existence of generalized Busemann functions with deterministic tilt vectors and coalescing geodesics.

\begin{theorem}\label{thm:Bhat-B}
Let $\Bhat\in\cKhatcoal$ and  $h=\Ehat[\hhB(\Bhat)]$. Assume that $\Phat\{\what:\hhB(\Bhat,\what)=h\}=1$. Then there exists a forward measurable cocycle $B\in\cKcoalfor$ defined on the canonical space $\Omega$ such that $\Bhat(\what)=\Bus(\w(\what))$, $\Phat$-almost surely. 
\end{theorem}

Theorem \ref{thm:decomp} below generalizes this representation to an ergodic decomposition of each $\Bhat\in\cKhatcoal$.

\begin{remark}\label{rk:counter0}
Theorem \ref{thm:Bhat-B} states that the $\Phat$-almost sure condition $\hhB(\Bhat, \what)=\Ehat[\hhB(\Bhat)]$ is sufficient for the existence of an $\kS$-measurable version of $\Bhat$. This condition is also necessary: if $\Bhat\in\cKhat$ is $\kS$-measurable, then $\hhB(\Bhat)$ is measurable with respect to the $\sigma$-algebra $\cI$ of $\kS$-events invariant under the group of shifts $\That$. By the i.i.d.\ assumption \eqref{iid}, this $\sigma$-algebra is trivial, which forces $\hhB(\Bhat,\what)=\Ehat[\hhB(\Bhat)]$ $\Phat$-almost surely. Example \ref{ex:counter} below illustrates how this condition can fail.

On the other hand, Lemma \ref{lem:hsupdif}\eqref{hsupdif.c} asserts that if 
\[\Ehat[\hhB(\Bhat)] \in \bigl\{-\nabla\shape(\xi\sig):\xi\in\ri\Uset,\;\sigg\in\{-,+\}\bigr\},\]  
then we do have $\hhB(\Bhat)=\Ehat[\hhB(\Bhat)]$ $\Phat$-almost surely. Further discussion is given in Remarks \ref{rk:counter1} and \ref{rk:counter2} after we introduce the set $\cH$ in \eqref{cHdef}.
\end{remark}
We also record the following sufficient condition to verify the hypothesis of Theorem \ref{thm:Bhat-B}. The proof is deferred to Appendix \ref{app:aux}.
\begin{lemma}
  Let $\Bhat\in\cKhat$ such that for some $\hat u\in(\R_+\times\R_-)\cup(\R_-\times\R_+)\setminus\{0\}$, $\hhB(\Bhat)\cdot \hat u=\Ehat[\hhB(\Bhat)]\cdot\hat u$, $\Phat$-almost surely. Then $\hhB(\Bhat)=\Ehat[\hhB(\Bhat)]$, $\Phat$-almost surely.\label{lem:ergosuf}
\end{lemma}

The second main result is strong uniqueness on the canonical space. Note that because the canonical space is shift-ergodic, if $\Bus \in \cKcoal$, then $\hhB(\Bus)$ is necessarily deterministic.

\begin{theorem}\label{thm:uniqueness}
Let $\Bus_1,\Bus_2\in\cKcoal$. Then exactly one of the following three happens: 
\begin{enumerate}  [label={\rm(\alph*)}, ref={\rm\alph*}] 
\item $\hhB(\Bus_1)\precneq\hhB(\Bus_2)$ and $\P(\Bus_1\precneq\Bus_2)=1$,
\item $\hhB(\Bus_1)\succneq\hhB(\Bus_2)$ and $\P(\Bus_2\succneq\Bus_1)=1$, or 
\item $\hhB(\Bus_1)=\hhB(\Bus_2)$ and $\P(\Bus_1=\Bus_2)=1$.
\end{enumerate}
\end{theorem}

 Denote by 
\begin{align}\label{cHdef}
\cH=\{\hhB(\Bus):\Bus\in\cKcoal\}
\end{align}
the collection of \textit{tilts} associated to shift-covariant recovering cocycles with coalescing geodesics which are defined on $\Omega$. Again, because the i.i.d.\ $\bbP$ is shift-ergodic, $\hhB(\Bus)$ is non-random for each  $\Bus \in \cKcoal$. 

\begin{proposition}
We have
\begin{align}
\{-\nabla \shape(\xi\sigg): \xi \in \ri\Uset, \,\sigg \in \{+,-\}\} \subset \sH \subset -\partial \shape(\Uset).\label{eq:Hinc}
\end{align}
\end{proposition}

\begin{proof}
The second inclusion comes from Lemma \ref{lem:hsupdif}\eqref{hsupdif.a}. For the first inclusion, take $\xi\in\ri\Uset$ and $\sigg\in\{-,+\}$. Let $h=-\nabla\shape(\xi\sigg)$. 
Theorem \ref{thm:exist} furnishes a cocycle $\Bhat\in\cKhat$ with $\Ehat[\hhB(\Bhat)]=h$. By Lemma \ref{lem:hsupdif}\eqref{hsupdif.c}, $\hhB(\Bhat)=h$, $\Phat$-almost surely. Theorem \ref{DLR-coal} says $\Bhat\in\cKhatcoal$.
 Then Theorem \ref{thm:Bhat-B} gives a $B\in\cKcoal$ with $\hhB(B)=h$, which implies that $h\in\cH$.
\end{proof}

By concavity, the function $\shape$ is differentiable at all but countably many points $\xi\in\ri\Uset$. By \eqref{eq:Martin}, the set on the left-hand side of \eqref{eq:Hinc} has at least countably infinitely many elements.

\begin{remark}
It is an open question whether either one of the inclusions in \eqref{eq:Hinc} is strict. 
\begin{enumerate}  [label={\rm(\alph*)}, ref={\rm\alph*}] 
\item If $\shape$ is differentiable, as is widely expected under assumption \eqref{iid}, $\partial \shape(\Uset)$ consists of genuine gradients $\nabla \shape(\xi)$ and then 
both inclusions in \eqref{eq:Hinc} are equalities. As mentioned in Remark \ref{rem:dir}, if $\xi$ is a direction of non-differentiability, then  the   super-differential $\partial\shape(\xi)$ is a nondegenerate line segment. 
\item If $h\in-\ri\partial\shape(\xi)$ for some $\xi\in\ri\Uset$ exists that is associated to an ergodic cocycle with coalescing geodesics (an element of $\cKhatcoal$   with $\hhB(\Bhat)$ deterministic), then the first inclusion in \eqref{eq:Hinc} is strict. 
\item The second inclusion is an equality if and only if every element of the negative of the super-differential is associated to an ergodic cocycle with coalescing geodesics. Then we say   there are no \textit{gaps} in $\sH$. Brito and Hoffman \cite{Bri-Hof-21} give an example of an ergodic FPP model where the limit shape is a diamond ($\ell^1$ ball). In this case there are only four semi-infinite geodesics rooted at the origin and the associated $\sH$ consists of the four gradients corresponding to the linear segments of the shape.  This provides an example where \textit{gaps} in   $\cH$ exist.
\end{enumerate}
\end{remark}

\begin{example}\label{ex:counter}
    We provide here an example where $\hhB(\Bhat)$ is not deterministic.
    Take distinct $h',h''\in\cH$. Then there exist $B^{h'},B^{h''}\in\cKcoal$ such that $\hhB(B^{h'})=h'$ and $\hhB(B^{h''})=h''$. Take $t\in(0,1)$. Let $Z$ be a Bernoulli($t$) random variable defined on the space $\{0,1\}$. Let $\Omhat=\Omega\times\{0,1\}$, equipped with the product Borel $\sigma$-algebra $\kShat$ and the shifts $\That_x(\w,Z)=(T_x\w,Z)$, $x\in\Z^2$. Let $\Phat$ be the distribution of $\what=(\w,Z)$. Let $\Bhat(\what)=Z B^{h'}(\w)+(1-Z)B^{h''}(\w)$.
    Then, the definition \eqref{h-def} gives $\hhB(\Bhat,\what)=Zh'+(1-Z)h''$. \end{example}

\begin{remark}\label{rk:counter1}
Theorem \ref{thm:Bhat-B} implies that for $\Bhat\in\cKhatcoal$ to have an $\kS$-measurable version, it is necessary that $\Ehat[\hhB(\Bhat)]\in\cH$. This condition can be made to fail by choosing $h',h''\in\cH$ in Example \ref{ex:counter} so that the line segment $[h',h'']$ is not entirely contained in $\cH$. The existence of such $h'$ and $h''$ follows from the asymptotics \eqref{eq:Martin}, which show that $\partial\shape(\Uset)$ cannot be a single line segment.
\end{remark}

\begin{remark}\label{rk:counter2}
Following up on Remark \ref{rk:counter0}, suppose there exists an $h\in\cH\setminus\bigl\{-\nabla\shape(\xi\sigg) : \xi \in \ri \Uset, \sigg\in\{+,-\}\bigr\}$. Then, by Lemma 4.6(a) in \cite{Jan-Ras-20-aop}, $h$ is in the interior of a line segment $[-\nabla\shape(\xi-),-\nabla\shape(\xi+)]$ for some $\xi\in\ri\Uset$ at which $\shape$ is not differentiable. In this case, there exists a $t\in(0,1)$ such that 
$h=th'+(1-t)h''$, where $h'=-\nabla\shape(\xi-)\precneq-\nabla\shape(\xi+)=h''$. Example \ref{ex:counter} now gives a $\Bhat\in\cKhatcoal$ with $\Ehat[\hhB(\Bhat)]=h\in\cH$ and yet $\hhB(\Bhat)\ne h$, $\Phat$-almost surely.
\end{remark}

By Theorem \ref{thm:uniqueness}, for each $h\in\cH$ there exists a unique $B^h\in\cKcoal$ such that $\hhB(B^h)=h$.
However, these cocycles are random variables defined up to null sets that may depend on $h$. We now construct a coupling of these random variables, which produces the tilt-indexed Busemann process on the canonical space $\Omega$. The index set of this process is $\cH\times\{-,+\}$. The process is obtained by taking left and right limits of the cocycles $B^h$ as $h$ ranges over a countable dense subset of $\cH$.

\begin{theorem}\label{thm:stexist} There exists a stochastic process 
\begin{align}
    \bigl(\Bus^{h\sig}(x,y) : x,y\in\bbZ^2,\, h \in \sH,\, \sigg \in \{+,-\}\bigr)
\end{align}
on $(\Omega, \sF, \bbP)$ with the following properties:
\begin{enumerate}  [label={\rm(\alph*)}, ref={\rm\alph*}]   \itemsep=3pt
\item\label{h-=h+} \textup{(}No $\pm$ distinction at fixed $h$\textup{)} For each $h \in \sH$, 
\[
\bbP\{\Bus^{h-}= \Bus^{h+}\}=1.
\]
When $\Bus^{h-}(\w) = \Bus^{h+}(\w)$, call the common cocycle $\Bus^h(\w)$.
\item\label{thm:stexist:Bus} \textup{(}Forward measurable Busemann function\textup{)} For each $h\in \sH, \Bus^h \in \cKcoalfor$.
\item \textup{(}Mean $-h$\textup{)} For each $h \in \sH$ and $i\in\{1,2\}$, $\bbE[\Bus^h(0,e_i)] = -h \cdot e_i.$
\item \textup{(}Monotonicity\textup{)} For $h,h'\in \sH$ with $h \cdot e_1 \leq h' \cdot e_1$, all $x \in \bbZ^2$, and $\bbP$ almost every $\w$
\[ 
    \Bus^{h-}(x,x+e_1) \geq  \Bus^{h+}(x,x+e_1) \geq  \Bus^{h'-}(x,x+e_1) \geq  \Bus^{h'-}(x,x+e_1)  
\]
and
\[ 
    \Bus^{h-}(x,x+e_2) \leq  \Bus^{h+}(x,x+e_2) \leq  \Bus^{h'-}(x,x+e_2) \leq  \Bus^{h'-}(x,x+e_2).
\]
\item \textup{(}Left-/right-continuity\textup{)} For $\bbP$-almost all $\w$, all $h\in \sH$, and all $x,y\in\Z^2$,
\begin{equation*}\begin{aligned}
\Bus^{h-}(x,y) = \lim_{\substack{\sH\tsp\tsp\ni\tsp \tsp h' \to h\\ h'\cdot e_1 \nearrow h \cdot e_1}} \Bus^{h\pm}(x,y) \,\text{ and }\, \Bus^{h-}(x,y) = \lim_{\substack{\sH\tsp\tsp\ni\tsp \tsp h' \to h\\ h'\cdot e_1 \searrow h \cdot e_1}} \Bus^{h\pm}(x,y).
\end{aligned}\end{equation*}
\end{enumerate}
Moreover, this process is unique in the sense that any two processes satisfying the above conditions are equal almost surely.
\end{theorem}

\begin{remark}
If one instead wishes to work with the (potentially) smaller direction-indexed Busemann process coming from restricting to $\{-\nabla \shape(\xi\pm) : \xi\in\Uset\}\subset \sH$ or with finite dimensional distributions (i.e. finitely many $h\in\sH$) of the process, strong uniqueness still holds.
\end{remark}

\begin{remark}
The shift-covariance 
implicitly contained in part \eqref{thm:stexist:Bus} of Theorem \ref{thm:stexist} is inherited by the full process. In particular, we have that
\[
(\Bus^{\aabullet}(x,y) : x,y\in\bbZ^2) \circ T_z = (\Bus^{\aabullet}(z+x,z+y) : x,y\in\bbZ^2) 
\]
$\bbP$-almost surely. 
\end{remark}

As a consequence of the above observation and the fact that $\bbP$ is an i.i.d.~measure on $\Omega=\bbR^{\bbZ^2}$, we have strong mixing.
\begin{corollary}
Call $\Bus^{\aabullet} = (\Bus^{\aabullet}(x,y) : x,y\in\bbZ^2)$ and let $z\in\bbZ^2\backslash\{0\}$. Then $\Bus^{\aabullet}$ is strongly mixing under $T_z$. Explicitly, this means that for all events $A \in \sF$ and Borel $C$,
\[
\lim_{n\to\infty}\bbP(A, \Bus^{\aabullet}\circ T_z^{-n}\in C ) = \bbP(A)\tspb\bbP(\Bus^{\aabullet}\in C). \qedhere
\]
\end{corollary}

 Combining Theorems \ref{thm:Bhat-B}, \ref{thm:uniqueness}, and \ref{thm:stexist} yields the following ergodic decomposition of the elements of $\cKhatcoal$. It implies that if $\hhB(\Bhat)$ is not deterministic, then the distribution of $\Bhat$ is a convex mixture of the distributions of cocycles in $\cKcoalfor$ coming from the tilt-indexed Busemann process.

\begin{theorem}\label{thm:decomp}
    Let $\Bhat\in\cKhatcoal$. Then  for $\Phat$-almost every $\what$, $\hhB(\Bhat,\what)\in\cH$ and  
    \[\Bhat(\what)=B^{\hhB(\Bhat,\what)-}(\w(\what))=B^{\hhB(\Bhat,\what)+}(\w(\what)).\]
\end{theorem}


\begin{remark}
Theorem \ref{thm:decomp} can be interpreted as another strong uniqueness statement. Any covariant recovering cocycle with coalescing geodesics on any space satisfying the hypotheses of Section \ref{sec:probsp} can be written as a mixture of elements of the (strongly unique) tilt-indexed Busemann process. Any such mixture is entirely captured by the law of the tilt vector $\hhB(\Bhat)$, defined in \eqref{h-def}, which encodes the dependence of the process on the $\That$-invariant $\sigma$-algebra $\widehat{\sI}$. The shift-covariance forces this mixture to assign zero mass to the exceptional tilts at which the Busemann process has discontinuities.
\end{remark}

\subsection{Connections with covariant coalescing systems of geodesics}
Our remaining results concern the connection between generalized Busemann functions with coalescing geodesics and covariant systems of coalescing geodesics. The proofs of the results in this section are contained in Section \ref{sec:equiv}.

Our first result is one of independent interest. Our next result says that if $\Bhat$ is a  shift-covariant recovering $L^1$ cocycle and  the  geodesics $\cgeod{}^{\Bhat,\aabullet}$    defined with the $e_1$ tiebreaking rule \eqref{eq:e1tie} coalesce, then these geodesics are the unique locally rightmost $\Bhat$-geodesics.
 Ties in a  cocycle can occur even if no two distinct paths with the same endpoints have equal weights. This result is therefore nontrivial even under the assumption that $\w_0$ has a continuous distribution.
\begin{lemma}\label{phih-unique}
    Take $\Bhat\in\cKhatcoal$. For $\Phat$-almost every $\what$ and each $u\in\Z^2$, $\cgeod{}^{\Bhat,u}(\what)$ is the only $\Bhat$-geodesic in the tree $\Geo_u^{\w(\what)}$.
\end{lemma}

Our final result in this section gives the equivalence between studying covariant systems of coalescing geodesics and studying shift covariant recovering $L^1(\bbP)$ cocycles. The first part of the next result says that if one uses a coalescing system of covariant geodesics to generate a cocycle, then that cocycle can in turn be used to recover the same family of locally rightmost geodesics by using the local rule in \eqref{eq:e1tie}. In light of the previous uniqueness result, this is not surprising. 

 Recall that $(\Omhat, \kShat,\Phat,\That)$ is ergodic if, for every event $H\in\kShat$ with $\That_x H=H$ for all $x\in\Z^2$, we have $\Phat(H)\in\{0,1\}$. In particular,   $(\Omega,\sF,\P,T)$ is ergodic by the i.i.d.\ assumption. The second part of the statement of the next result says that if $(\Omhat, \kShat,\Phat,\That)$ is ergodic, then any shift-covariant coalescing family of geodesics defines a recovering shift-covariant $L^1$ cocycle. The  $L^1$ property of the cocycle  is not obvious and is false if the ergodicity assumption is dropped. This fact is  recorded in Remark \ref{rem:ergonec} following the theorem.  The proof of the theorem relies on our strong existence result, Theorem \ref{thm:stexist}, and Martin's shape asymptotic  \eqref{eq:Martin}. 

\begin{theorem}
Let $(\Omhat, \kShat,\Phat,\That)$ satisfy the hypotheses of Section \ref{sec:probsp}. Let $\pihat^\aabullet\in\Gcchat$ and let $\Bhat=\Busgeo_{\pihat^0}$. 
Then $\Phat\{\forall u\in\Z^2:\cgeod{}^{\Bhat,u}=\pihat^u\}=1$.
Moreover, if $(\Omhat, \kShat,\Phat,\That)$ is ergodic, then   $\pihat^\aabullet\in\Gcchat$ implies $\Busgeo_{\pihat^0}\in\cKhatcoal$.\label{thm:xB}
\end{theorem}
\begin{remark}
By taking mixtures, the $L^1$ portion of the claim that $\Busgeo_{\pihat^0}\in\cKhatcoal$ in the second part of Theorem \ref{thm:xB} is false without the ergodicity assumption. This is because the asymptotics of $\shape$ in \eqref{eq:Martin} show $\E[B^{-\nabla\shape(\xi-)}(0,e_1)]=\nabla\shape(\xi-)\cdot e_1\to\infty$ as $\xi\to e_2$ and $\E[B^{-\nabla\shape(\xi+)}(0,e_2)]=\nabla\shape(\xi+)\cdot e_2\to\infty$ as $\xi\to e_1$. \label{rem:ergonec}
\end{remark}
With our main results stated, we now turn to the proofs, beginning with strong existence.
\section{Strong existence and uniqueness of generalized Busemann functions}\label{sec:unique}

The main aim of the first part of this section is to start with a given generalized Busemann function on $\Omhat$ with coalescing geodesics, $\Bhat\in\cKhatcoal$, construct an $\kS$-measurable version of its coalescing geodesics, and subsequently use this to obtain an $\kS$-measurable version of $\Bhat$ itself. While one could try to rely on general measure-theoretic techniques (e.g., sections) to produce a system of coalescing geodesics---similar to the approach used in \cite[Theorem 3.2]{Jan-Ras-20-aop} to establish the $\P$-a.s.\ existence of directed geodesics after proving their $\Phat$-a.s.\ existence---the challenge lies in maintaining shift-covariance across all $T_x\w$. 

The basic idea, following a similar construction in \cite{Ahl-Hof-16-}, is to apply a variant of the classical inverse CDF sampling method to the conditional distribution of a $\Bhat$-geodesic, given $\kS$. Because the path space is totally ordered (see Figure \ref{fig:geotree}), we can define quantile functions and this method works essentially the same way as for real random variables. The outcome is a process of random geodesics which is shift-covariant by construction. We then show that under appropriate hypotheses on $\Bhat$, this process is essentially constant and defines a family of coalescing geodesics, which then generate a Busemann function that equals $\Bhat$ almost surely. Strong uniqueness comes from showing that there is a total ordering on such objects indexed by the tilt vector and so in particular any two generalized Busemann functions with the same deterministic tilt vector must be equal. Extending these properties to the full process is essentially immediate from monotonicity.

We begin by discussing left- and right-isolated geodesics. These play a central role in the argument.

\subsection{Isolated geodesics}
Since $\Geo_u^\w$ is totally ordered and compact, it has the greatest lower bound and least upper bound properties. For $u\in\Z^2$, let $m=u\cdot(e_1+e_2)$ and define these collections of semi-infinite geodesics rooted at $u$: 
\be\begin{aligned}
\LI_u^\w&=\bigcup_{ \substack{\sigma_{m:n} \text{ up-right, } n \tsp\in\tsp \Z_{\ge m}\\[3pt] \exists \tsp\pi\tsp\in\tsp \Geo_u^\w : \,\pi_{m:n}=\sigma_{m:n}}}\bigl\{\inf\{\pi \in \Geo_{u}^\w : \pi_{m:n}=\sigma_{m:n}\}\bigr\}\quad\text{and}\\
\RI_u^\w&=\bigcup_{ \substack{\sigma_{m:n} \text{ up-right}, n \tsp\in\tsp \Z_{\ge m}\\[3pt] \exists \tsp\pi\tsp\in\tsp \Geo_u^\w : \, \pi_{m:n}=\sigma_{m:n}}}
\bigl\{\sup\{\pi \in \Geo_{u}^\w : \pi_{m:n}=\sigma_{m:n}\}\bigr\}.
\end{aligned}\label{eq:isodense}
\ee 
Note that the conditions in the unions above and the locally-rightmost semi-infinite geodesics appearing in \eqref{eq:isodense} (e.g.\ $\sup\{\pi \in \Geo_{u}^\w : \pi_{m:n}=\sigma_{m:n}\}$ for some finite up-right path $\sigma_{m:n}$ with $\sigma_m=u$) are $\kSfor_u$-measurable, because they can be constructed inductively from the arrows $\edge_{u,x}^i(\w)$ described at the beginning of Appendix \ref{app:aux}.

A nontrivial locally-rightmost semi-infinite geodesic $\lambda\in\Geo_u^\w$ is \emph{left-isolated} if it is not a strictly increasing  limit of members of $\Geo_u^\w$ from the left, equivalently, $\lambda\succneq \sup\{\gamma \in \Geo_{u}^\w : \gamma\precneq \lambda\}$.  Analogously, a nontrivial $\rho\in\Geo_u^\w$ is \emph{right-isolated} if $\rho \precneq \inf\{\gamma \in \Geo_u^\w : \gamma \succneq \rho\}$. A nontrivial locally-rightmost semi-infinite geoddesic is right-isolated if and only if  it is in $\RI_u^\w$ and it is left-isolated if and only if  it is in $\LI_u^\w$.  To see this, take $\lambda \in \Geo_u^\w$ and consider the first index $n$ at which $\lambda$ differs from the supremum of locally-rightmost semi-infinite geodesics which are strictly less than it. Then $\lambda$  is the infimum of all locally-rightmost semi-infinite geodesics which contain the segment $\lambda_{u:n}$. The right-isolated case is similar. Note  that $u+\Z_+e_2\in\LI_u^\w$ and $u+\Z_+e_1\in\RI_u^\w$ and also that $\RI_u^\w$ is right-dense in $\Geo_u^\w$ and $\LI_u^\w$ is left-dense in $\Geo_u^\w$. See Figure \ref{fig:geotree}.

\begin{figure}[h!]
\centering
\includegraphics[width=.5\textwidth]{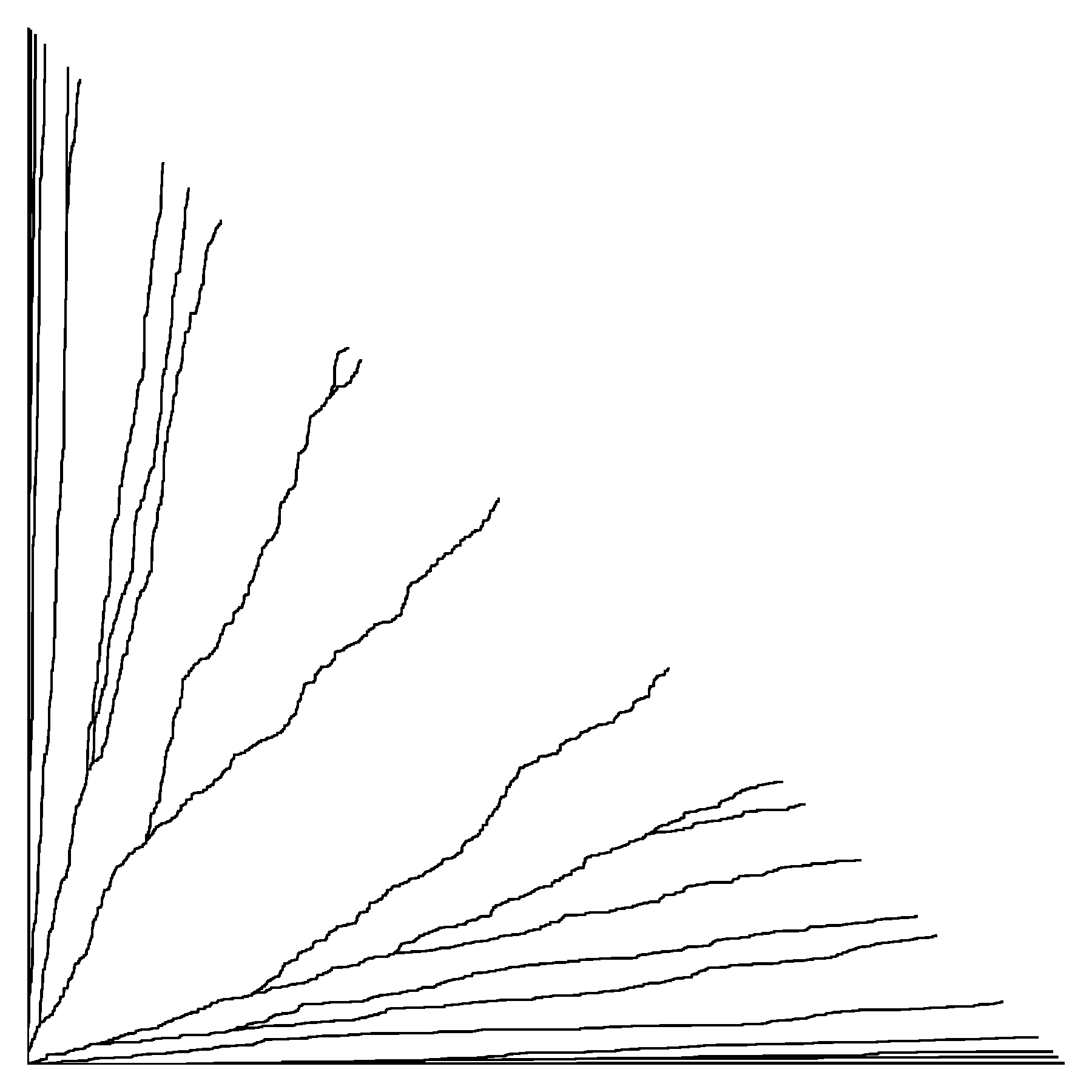}
\caption{A simulation of an approximation to the tree $\Geo_0^\w$ of infinite geodesics rooted at the origin in the i.i.d.~Exponential(1) model. Right-isolated geodesics are those which cannot be approximated from the right (in the topology in which paths converge if they are eventually equal on finite sets) in this tree. Graphically, such geodesics appear as the left boundary of one of the connected components of the quadrant $\bbZ_+^2$ which contain no vertices on any geodesics of $\Geo_0^\w$. Left-isolated geodesics are the right boundary of such a region.}\label{fig:geotree}
\end{figure}

Let 
    \[\Geo^\w=\bigcup_{u\tsp\in\tsp\Z^2}\Geo_u^\w\]
be the collection of all locally-rightmost semi-infinite geodesics from all  starting points.

On the union $\Geo^\w$,   $\pi\apreceq\gamma$ and $\gamma\apreceq\pi$ happen together  if and only if  $\pi\coal\gamma$.  In this situation, after their first meeting the two remain together, again by virtue of the uniqueness of rightmost point-to-point  geodesics.

\subsection{Strong existence}
Let $\nu^{\w}$ denote the conditional distribution of $\Phat$ on $(\Omhat, \kShat)$ given $\kS$. Recall the countable right-dense set of $\kSfor_u$-measurable right-isolated geodesics $\RI_u^\w \subset\Geo_u^\w$, defined in \eqref{eq:isodense}, and note that by definition $\RI_u^{T_z\w} = \RI_{u+z}^\w$ for all $u,z\in\Z^2$. 

For $u\in\Z^2$ and $s \in \bbQ\cap [0,1]$, define an  $\kS$-measurable locally-rightmost semi-infinite geodesic $\cdfpi^{\Bhat,u,s} \in \Geo_u^\w$ via
\begin{align}
\cdfpi^{\Bhat,u,s}(\what)=\inf\bigl\{\rho\in\RI_u^{\w(\what)}:\nu^{\w(\what)}(\cgeod{}^{\Bhat,u}\preceq \rho) \geq s\bigr\}.\label{eq:inf-formula-rational}
\end{align}
The set in \eqref{eq:inf-formula-rational} is not empty because it contains the trivial geodesic $u+\Z_+e_1$. We took the infimum in \eqref{eq:inf-formula-rational} over the countable set $\RI_u^{\w(\what)}$  to ensure measurability. We verify that this expression is measurable in Corollary \ref{cor:infmeas} below. 

Extend this definition to $s \in [0,1]$ by setting 
\begin{align}
\cdfpi^{\Bhat,u,s}(\what)= \sup\bigl\{\cdfpi^{\Bhat,u,r}(\what) : r \in \bbQ\cap[0,1], \,r <  s\bigr\}.\label{eq:inf-formula}
\end{align}
Denote by $\cdfpi^{\Bhat,u,\aabullet}$ the process $\bigl(\cdfpi^{\Bhat,u,s}:s\in[0,1]\bigr)$, defined through \eqref{eq:inf-formula}. In Lemma \ref{lm:xBhatproc}\eqref{lm:xBhatproc.defined}, we show that \eqref{eq:inf-formula-rational} and \eqref{eq:inf-formula}  agree on rational $s$.  In Lemma \ref{lm:xBhatproc}\eqref{lm:xBhatproc.inf}, we show that the infimum can be taken over the uncountable tree $\Geo_u^{\w(\what)}$ for any $s\in[0,1]$.

Since $\Geo_u^\w$ is closed, we have, for $\Phat$-almost every $\what$, $\cdfpi^{\Bhat,u,s}(\what)\in\Geo_u^{\w(\what)}$ for all $s\in[0,1]$. In the statement of the next result, $\Skor([0,1],\bbX_u)$ and $\Skor([0,1],\Geo_u^\w)$ are Skorokhod spaces of left-continuous paths with right limits (see \cite[Section 3.5]{Eth-Kur-86} for a definition of the Skorokhod topology) taking values in the compact metric spaces $\bbX_u$ and $\Geo_u^\w$, respectively.

In what follows, the phrase ``for $\Phat$-almost every $\what$'' means  the existence of an $\kShat$-measurable event of full $\Phat$-measure such that the stated property holds for each $\what$ in this event. By taking intersections of all shifts by $\{\That_x : x \in\bbZ^2\}$, this event can without loss of generality be assumed to be shift invariant.

\begin{lemma}\label{lm:xBhatproc}
For each $\Bhat\in \cKhat$, the process $\cdfpi^{\Bhat,u,\aabullet}$  satisfies the following properties: 
\begin{enumerate}  [label={\rm(\alph*)}, ref={\rm\alph*}]   \itemsep=3pt
\item\label{lm:xBhatproc.defined} 
For $\Phat$-almost every $\what$, the definition in \eqref{eq:inf-formula} agrees with \eqref{eq:inf-formula-rational} for rational $s$, so $\cdfpi^{\Bhat,u,\aabullet}(\what)$ is well-defined.
\item\label{lm:xBhatproc.meas} $\what\mapsto \cdfpi^{\Bhat,u,\aabullet}(\what)$ is an $\kS$-measurable, $\Skor([0,1],\bbX_u)$-valued random variable which almost surely takes values in $\Skor([0,1],\Geo_u^{\w(\what)})$.
\item\label{lm:xBhatproc.inf} 
For $\Phat$-almost every $\what$ and all $s \in [0,1]$,
\begin{align}
\cdfpi^{\Bhat,u,s}(\what)=\inf\bigl\{\gamma\in\Geo_u^{\w(\what)}:\nu^{\w(\what)}(\cgeod{}^{\Bhat,u}\preceq \gamma)\geq s\bigr\}.\label{eq:fullinf}
\end{align}
Consequently, $\cdfpi^{\Bhat,u,r}(\what)\preceq\cdfpi^{\Bhat,u,s}(\what)$ for $r\le s$ in $[0,1]$.
\item \label{lm:xBhatproc.key} 
For $\Phat$-almost every $\what$ and all $\gamma \in \Geo_u^{\w(\what)}$,
\begin{align}
\bigl\{s\in [0,1] : \cdfpi^{\Bhat,u,s}(\what) \preceq \gamma\bigr\}=\bigl[0,\nu^{\w(\what)}(\cgeod{}^{\Bhat,u} \preceq \gamma)\bigr]. 
\label{eq:key-sym} 
\end{align}
\item \label{lm:xBhatproc:cov} 
For $\Phat$-almost every $\what$, all $u,z\in\Z^2$, and all $s\in[0,1]$,
\be\begin{aligned}\label{x-cov}
\cdfpi^{\Bhat,u+z,s}(\what)=z+\cdfpi^{\Bhat,u,s}(\That_z\what).
	\end{aligned}\qedhere\ee
\end{enumerate}
\end{lemma}

\begin{proof}
First note that if $r < s$ and both are rational, then using definition \eqref{eq:inf-formula-rational}, we have $\cdfpi^{\Bhat,u,r} \preceq \cdfpi^{\Bhat,u,s}$.  Monotonicity and the fact that $\Geo_u^\w$ has the least upper bound and greatest lower bound properties imply the existence of left and right limits of these paths, which lie in $\Geo_u^\w\subset \bbX_u$.

We check that \eqref{eq:fullinf} holds for rational $s$, with the definition in \eqref{eq:inf-formula-rational}. 
For such $s$, we have
\[
\cdfpi^{\Bhat,u,s} = \inf\bigl\{\rho\in\RI_u^\w:\nu^\w(\cgeod{}^{\Bhat,u}\preceq \rho) \geq s\}  \succeq \inf\bigl\{\gamma\in\Geo_u^\w:\nu^\w(\cgeod{}^{\Bhat,u}\preceq \gamma)\geq s\bigr\} = \widetilde{\cdfpi}^{s},
\]
because the set in the infimum on the left is a subset of the one on the right. Continuity of probability implies that 
\begin{align}\label{auxi}
\nu^\w(\cgeod{}^{\Bhat,u} \preceq \widetilde{\cdfpi}^{s}) \geq s.
\end{align}
Either $\widetilde{\cdfpi}^{s} \in \RI_u^\w$, in which case we have $\widetilde{\cdfpi}^{s} = \cdfpi^{\Bhat,u,s}$, or there exists a sequence $\rho^n\in \RI_u^\w$ with $\rho^n \searrow \widetilde{\cdfpi}^{s}$. But then we must have $\nu^\w(\cgeod{}^{\Bhat,u} \preceq \rho^n) \geq s$ by monotonicity, which implies that $\cdfpi^{\Bhat,u,s} \preceq\inf\{\rho^n : n \in \N\} = \widetilde{\cdfpi}^{s}$ and again $\widetilde{\cdfpi}^{s}=\cdfpi^{\Bhat,u,s}$.

Next, observe that if $s$ is any number in $[0,1]$ for which \eqref{eq:fullinf} holds, then for all $\gamma\in\Geo_u^\w$
\begin{align}
\cdfpi^{\Bhat,u,s} \preceq \gamma \Longleftrightarrow s \leq \nu^\w(\cgeod{}^{\Bhat,u}\preceq \gamma).\label{eq:fullinfproof}
\end{align}
$\Leftarrow$ comes from \eqref{eq:fullinf} and $\Rightarrow$ comes from \eqref{auxi} and that $\widetilde{\cdfpi}^{s}=\cdfpi^{\Bhat,u,s}_{0:\infty}$. In particular, since we showed above that \eqref{eq:fullinf} holds for all rational $s\in[0,1]$, we now know that \eqref{eq:fullinfproof} holds for all such $s$.

With this observation in mind, we check that the process $\cdfpi^{\Bhat,u,\aabullet}$ is well-defined, i.e., that with the definition in \eqref{eq:inf-formula-rational}, we have left-continuity over the rational $s\in[0,1]$. For $s\in [0,1]$, call
\[
\overline{\cdfpi}^{s} = \sup\bigl\{\cdfpi^{\Bhat,u,r} : r<s , r \in \bbQ\cap[0,1]\bigr\}.
\] 
The monotonicity observed at the beginning of the proof implies that if $s\in[0,1]$ is rational, then $\overline{\cdfpi}^s\preceq \cdfpi^{\Bhat,u,s}$. On the other hand,  from $\nu^\w(\cgeod{}^{\Bhat,u} \preceq \overline{\cdfpi}^s)\geq \nu^\w(\cgeod{}^{\Bhat,u} \preceq \cdfpi^{\Bhat,u,r}) \geq r$ for all rational $r$ with $r<s$,  we see that $\nu^\w(\cgeod{}^{\Bhat,u} \preceq \overline{\cdfpi}^s) \geq s$, which implies $\cdfpi^{\Bhat,u,s} \preceq \overline{\cdfpi}^s$ by \eqref{eq:fullinfproof}. 
 This is left-continuity on the rationals and part \eqref{lm:xBhatproc.defined} follows. 

The definition \eqref{eq:inf-formula} is left-continuous with right limits by construction and monotonicity on the rationals. $\kS$-measurability then follows from the fact that the path $\cdfpi^{\Bhat,u,\aabullet}$ is determined by the values of $\cdfpi^{\Bhat,u,s}$ defined according to \eqref{eq:inf-formula-rational} for $s \in \bbQ \cap [0,1]$, the $\kSfor_u$-measurability of the paths in $\RI_u^\w$, and the $\kS$-measurability of $\nu^\w$.
Part \eqref{lm:xBhatproc.meas} is proved.

Next, we show that \eqref{eq:fullinf} holds also for irrational $s$. For $s$ irrational, we denote the infimum in \eqref{eq:fullinf} by $\widetilde{\cdfpi}^s$, similarly to what was done above. Arguing exactly as above, by continuity of probability, we have $\nu^\w(\cgeod{}^{\Bhat,u} \preceq \widetilde{\cdfpi}^s) \geq s$. Since we already proved \eqref{eq:fullinf}  for rationals, we get that for all $r$ rational with $r<s$, $\cdfpi^{\Bhat,u,r} \preceq \widetilde{\cdfpi}^s$. Then by \eqref{eq:inf-formula}, $\cdfpi^{\Bhat,u,s} \preceq \widetilde{\cdfpi}^s$. 

On the other hand, we already showed that $\nu^\w(\cgeod{}^{\Bhat,u} \preceq \cdfpi^{\Bhat,u,r}) \geq r$ holds for all rational $r\in[0,1]$. Hence, $\nu^\w(\cgeod{}^{\Bhat,u} \preceq \cdfpi^{\Bhat,u,s}) \geq \nu^\w(\cgeod{}^{\Bhat,u} \preceq \cdfpi^{\Bhat,u,r}) \geq r$ for all rational $r<s$, which implies that $\nu^\w(\cgeod{}^{\Bhat,u} \preceq \cdfpi^{\Bhat,u,s})\geq s$ and so the reverse inequality $\cdfpi^{\Bhat,u,s} \succeq \widetilde{\cdfpi}^s$ also holds. Part \eqref{lm:xBhatproc.inf} is proved, which in turn implies that \eqref{eq:fullinfproof} holds for all $s\in[0,1]$ and proves part \eqref{lm:xBhatproc.key}.

We verify \eqref{lm:xBhatproc:cov} using \eqref{eq:fullinf} for each $s$. First, we note that by definition, $\gamma \in \Geo_u^\w$ if and only if $z + \gamma\in \Geo_{u+z}^{T_{-z}\w}$, where the addition is understood as the translation $(z + \gamma)_i=z + \gamma_i$. Similarly, for $\gamma\in\bbX_{u+z}$, the shift covariance of $\Bhat$ in \eqref{covariance} implies that $\That_{-z}\{\what : \cgeod{}^{\Bhat,u+z}(\what) \preceq \gamma\} = \{\what : \cgeod{}^{\Bhat,u}(\what) \preceq \gamma-z\}$. Changing variables by  $\gamma' = \gamma-z$, we deduce part \eqref{lm:xBhatproc:cov}:
\begin{align*}
\cdfpi^{\Bhat,u+z,s}(\That_{-z}\what) &= \inf\bigl\{\gamma\in\Geo_{u+z}^{\w(\That_{-z}\what)}:\nu^{\w(\That_{-z}\what)}(\cgeod{}^{\Bhat,u+z}\preceq \gamma) \geq s\bigr\} \\
&=  z + \inf\bigl\{\gamma'\in\Geo_{u}^{\w(\what)}:\nu^{\w(\what)}(\cgeod{}^{\Bhat,u}\preceq \gamma') \geq s\bigr\} = z + \cdfpi^{\Bhat,u,s}(\what).
\qedhere \end{align*}
\end{proof}

Note that there is no dependence in $\cdfpi^{\Bhat,u,s}$ on $\what$ through the superscript $\Bhat$. The $\Bhat$ in the superscript is just to remind us that if we use a different cocycle we get a different path. Due to the $\kS$-measurability in Lemma \ref{lm:xBhatproc}\eqref{lm:xBhatproc.meas}, instead of $\cdfpi^{\Bhat,u,s}_{0:\infty}(\what)$ we write $\cdfpi^{\Bhat,u,s}(\w(\what))$ and frequently simplify it further to $\cdfpi^{\Bhat,u,s}(\w)$ because now these paths can be regarded as functions of either  $\what\in\Omhat$ or $\w\in\Omega$.

We illustrate the definition with a simple case. This example is not referred to in the sequel. 

\begin{example} 
Suppose that  the  conditional distribution $\nu^\w\{\xhat^{u,\Bhat}_{0:\infty}\in\acdot\tspa\}$ is supported on finitely many distinct geodesics $\pi^1 \precneq  \pi^2 \precneq \dotsm \precneq\pi^m$ with $a_k=\nu^\w\{\cgeod{}^{\Bhat,u}=\pi^k\}$, $b_0=0$,  and $b_k=a_1+\dotsm+a_k$, for $k\in\{1,\dotsc,m\}$. Then
\[   \nu^\w\{\cgeod{}^{\Bhat,u}\preceq \gamma\}= 
\begin{cases}   0,  &\text{if }\gamma \precneq \pi^{1},   \\  
b_k, &\text{if }\pi^k \preceq \gamma \precneq \pi^{k+1}\text{ for }1\le k<m,\text{ and}\\
1, &\text{if }\gamma \succeq \pi^{m}_{0:\infty}.   \end{cases} 
\] 
Thus, for $s \in [0,1]$,
\[
\cdfpi^{\Bhat,u,s}= \inf\bigl\{\gamma \in \Geo_u^\w : \nu^\w(\cgeod{}^{\Bhat,u} \preceq\gamma)\geq s\bigr\} = \begin{cases}
u + \Z_+ e_2&\text{if }s = 0\text{ and}\\
\pi^{k} &\text{if }b_{k-1} < s \leq b_k\text{ for }k \in\{1,\dots,m\}. 
\end{cases}\qedhere
\]
\end{example}

Recall that $\leb$ is the Lebesgue measure on $[0,1]$, but we write integrals with respect to this measure using the standard notation $\int_0^1 f(s)\,ds$.

\begin{lemma}\label{lm:Fubini-new}
Fix $\Bhat\in \cKhat$ and $u\in\Z^2$
 and define $\cdfpi^{\Bhat,u,\aabullet}$ as in Lemma \ref{lm:xBhatproc}. Equip $\Omega \times\bbX_u$  with the product Borel $\sigma$-algebra. The distribution on $\Omega\times\bbX_u$  of $(\w(\what),\cgeod{}^{\Bhat,u}(\what))$ under $\Phat(d\what)$ is the same as the distribution of
$(\w,\cdfpi^{\Bhat,u,s}(\w))$ under $\P(d\w)\otimes\leb(ds)$.
\end{lemma}

\begin{proof}
To prove the lemma, it suffices to show that for all $A
\in \sB(\bbX_u)$ and $V\in \sF$, 
\[
\Phat(\w \in V, \, \cgeod{}^{\Bhat,u}\in A) = \int_0^1\bbP(\cdfpi^{\Bhat,u,s}\in A, V)\,ds.
\]
For this to hold, it is sufficient to consider $A$ of the form $A = \{\gamma \in \bbX_u : \gamma \preceq \pi\}$ for fixed $\pi\in\bbX_u$ because the collection of sets of this form is closed under intersections and generates $\sB(\bbX_u)$. See Lemma \ref{lem:gen}.

Fix $\pi \in \bbX_u$. We then have
\begin{align*}
\Phat(\w \in V,\, \cgeod{}^{\Bhat,u} \preceq \pi) &= \Ehat\bigl[\nu^\w(\cgeod{}^{\Bhat,u}\preceq\pi)\one_V(\w)\bigr] = \E\bigl[\nu^\w(\cgeod{}^{\Bhat,u}\preceq \pi)\one_V(\w)\bigr].
\end{align*}
Call $\gamma^\w = \sup\{\lambda\in \LI_u^\w : \lambda \preceq \pi\}\in\Geo_u^\w$. Then $\gamma^\w$ is 
 $\kSfor_u$-measurable and the set inside the supremum is non-empty (as it always contains $u + \bbZ_+e_2$). Moreover, for any $\rho\in \Geo_u^\w$, $\rho \preceq \gamma^\w$ if and only if $\rho\preceq \pi$. 
 
 Working on the $\P$-almost sure event where $\nu^\w(\cgeod{}^{\Bhat,u}\in\Geo_u^\w)=1$, we may write 
\[\nu^\w(\cgeod{}^{\Bhat,u}\preceq \pi) = \nu^\w(\cgeod{}^{\Bhat,u}\preceq \gamma^\w).\]

By Lemma \ref{lm:xBhatproc}\eqref{lm:xBhatproc.key} and the Fubini-Tonelli theorem, we have
\begin{align*}
\E\bigl[\nu^\w(\cgeod{}^{\Bhat,u}\preceq \pi)\one_V(\w)\bigr] &= \E\bigl[\nu^\w(\cgeod{}^{\Bhat,u}\preceq \gamma^\w)\one_V(\w)\bigr] =\E\Bigl[\int_0^1 \one_V(\w)\one\{\cdfpi^{\Bhat,u,s}\preceq \gamma^\w\}\,ds\Bigr] \\
&= \int_0^1 \P(\cdfpi^{\Bhat,u,s} \preceq \gamma^\w,V)\,ds.
\end{align*}
For each $s$, restricting to the $\P$-almost sure event on which $\cdfpi^{\Bhat,u,s} \in \Geo_u^\w$, this last expression is equal to
\[
\int_0^1 \P(\cdfpi^{\Bhat,u,s} \preceq \pi,V)\,ds.
\]
The result now follows.
\end{proof}

Next we narrow the assumptions to include coalescence of $\Bhat$-geodesics. In the statement of the next result, for $u,v\in\bbZ^2$, we say that $(\pi,\gamma)\in\Geo_u^\w\times \Geo_v^\w$ is a \textit{coalescing pair} if $\pi \coal \gamma.$

\begin{lemma}\label{lem:coal-equiv}
Fix $\Bhat \in \cKhatcoal$. Then with $\cdfpi^{\Bhat,u,\aabullet}$ and $\cdfpi^{\Bhat,v,\aabullet}$ defined as in Lemma \ref{lm:xBhatproc}, for $\P$-almost every $\w$, for any $u,v\in\Z^2$, and for all coalescing pairs $(\pi,\gamma)\in\Geo_u^\w\times\Geo_v^\w$, 
\be
\bigl\{s\in[0,1]:\cdfpi^{\Bhat,u,s}(\w)\preceq \pi\bigr\} = \bigl\{s\in[0,1]:\cdfpi^{\Bhat,v,s}(\w)\preceq \gamma\bigr\}.\label{eq:coal-equiv}\qedhere
\ee
\end{lemma}

\begin{proof}
By $\Bhat\in\cKhatcoal$ and Fubini's theorem, the event $\Omhat_1=\{\what:\nu^\w(\cgeod{}^{\Bhat,u}\coal \cgeod{}^{\Bhat,v})=1\}$ has $\Phat$-probability one.
Let $\Omhat_2$ be the intersection of $\Omhat_1$ with the full $\Phat$-probability event in Lemma \ref{lm:xBhatproc}\eqref{lm:xBhatproc.key}.
Take $\w\in\Omhat_2$. Then for any $u,v\in\Z^2$, $(\pi,\gamma)\in\Geo_u^\w$ such that $\pi\coal\gamma$, 
we have
\begin{align*}
 \nu^\w(\cgeod{}^{\Bhat,v}\preceq\gamma)-\nu^\w(\cgeod{}^{\Bhat,u}\preceq\pi)
 &\le\nu^\w(\cgeod{}^{\Bhat,u} \succneq \pi,\, \cgeod{}^{\Bhat,v} \preceq \gamma) \leq 1 - \nu^\w(\cgeod{}^{\Bhat,u}\coal \cgeod{}^{\Bhat,v})=0
\end{align*}
and
\begin{align*}
 \nu^\w(\cgeod{}^{\Bhat,u}\preceq\pi)-\nu^\w(\cgeod{}^{\Bhat,v}\preceq\gamma)
 &\le\nu^\w(\cgeod{}^{\Bhat,u} \preceq \pi,\, \cgeod{}^{\Bhat,v} \succneq \gamma) \leq 1 - \nu^\w(\cgeod{}^{\Bhat,u}\coal \cgeod{}^{\Bhat,v})=0.
\end{align*}
Thus, $\nu^\w(\cgeod{}^{\Bhat,u} \preceq \pi)=\nu^\w(\cgeod{}^{\Bhat,v}\preceq \gamma)$. Lemma \ref{lm:xBhatproc}\eqref{lm:xBhatproc.key} then gives \eqref{eq:coal-equiv} for all $\what\in\Omhat_2$ and coalescing $\pi\in\Geo_u^\w$ and $\gamma\in\Geo_v^\w$.
\end{proof}

\begin{lemma}\label{lm:coal}
Fix $\Bhat \in \cKhatcoal$. For all $u,v \in \bbZ^2$, with $\cdfpi^{\Bhat,u,\aabullet}$ and $\cdfpi^{\Bhat,v,\aabullet}$ defined as in Lemma \ref{lm:xBhatproc},
\[
\int_0^1 \P(\cdfpi^{\Bhat,u,s}\coal \cdfpi^{\Bhat,v,s})\,ds=1. \qedhere
\]
\end{lemma}

\begin{proof}
Fix $u,v\in\Z^2$.  Using Lemma \ref{lm:Fubini-new},
\begin{align*}  
1&=\Phat\{ \tspa \cgeod{}^{\Bhat,u}\coal \cgeod{}^{\Bhat,v} \tspa \} 
\le  \Phat\bigl\{ \what: \tspa \exists \pi \in \Geo^{\w(\what)}_v \text{ s.t. }  \pi \coal \cgeod{}^{\Bhat,u} \tspa \bigr\}   \\
&=\int_0^1 \P\bigl\{ \tspa \w: \tspa \exists \pi \in \Geo^{\w}_v \text{ s.t. }  \pi \coal \cdfpi^{\Bhat,u,s} \tspa \bigr\}\,ds.
\end{align*} 
By symmetry, switching the roles of $u$ and $v$,
\[  1= \int_0^1 \P\bigl\{ \tspa \w: \tspa \exists \pi \in \Geo^{\w}_u \text{ s.t. }  \pi \coal \cdfpi^{\Bhat,v,s} \tspa \bigr\}\,ds. 
\]
Thus there exists a measurable set $D_0\subset\Omega\times(0,1]$ 
such that  $\P\otimes\leb(D_0)=1$ and $ \forall (s,\w)\in D_0$:
\be\label{aux425}  
 \exists \gamma \in \Geo^{\w}_v \text{ s.t. }  \gamma \coal \cdfpi^{\Bhat,u,s}
\quad\text{and}\quad 
\exists \gamma' \in \Geo^{\w}_u \text{ s.t. }  \gamma' \coal \cdfpi^{\Bhat,v,s}. 
\ee

By Lemma \ref{lem:coal-equiv}, there exists an event $\Omega_0\in\kS$ with $\P(\Omega_0)=1$ and on which \eqref{eq:coal-equiv} holds for all coalescing pairs  $(\pi,\gamma) \in \Geo_u^\w\times \Geo_v^\w$.
We claim that for $(s,\w)\in D_0\cap(\Omega_0\times(0,1])$, $\cdfpi^{\Bhat,u,s}(\w)\coal\cdfpi^{\Bhat,v,s}(\w)$. 

By \eqref{aux425} there is a coalescing pair $(\cdfpi^{\Bhat,u,s}, \gamma)\in\Geo^{\w}_u\times\Geo^{\w}_v$. Then by \eqref{eq:coal-equiv},   $\cdfpi^{\Bhat,v,s}\preceq \gamma$.  Thus $\cdfpi^{\Bhat,v,s}\apreceq \cdfpi^{\Bhat,u,s}$. A symmetric argument gives  $\cdfpi^{\Bhat,u,s}\apreceq \cdfpi^{\Bhat,v,s}$. Hence, the two  coalesce.
\end{proof}

Because geodesics proceed up-right in directed last-passage percolation, it is natural  to expect that a shift-covariant family of coalescing geodesics will be measurable with respect to the weights ahead of the root.   
The next result records this fact if the weights are i.i.d. We state this result on $\Omega$. An analogous statement holds on $\Omhat$ with $\sFfor_u$ replaced by $\kSfor_u$ if the geodesics are $\kS$ measurable.

\begin{lemma}\label{lm:pi-forward} 
Let $\pi^u:\Omega\to\pathsp_u$, $u\in\Z^2$, be measurable maps such that $\pi^u\in\Geo_u^\w$ $\P$-almost surely. Suppose the family $(\pi^u)_{u\in\Z^2}$ is shift-covariant, 
\[
\P\{\pi^u=u+\pi^0\circ T_u\}=1 \quad \text{for all } u\in\Z^2,
\]
and that $\pi^u\coal\pi^v$ holds with probability one for all $u,v\in\Z^2$. Then each $\pi^u$ is $\sFfor_u$-measurable up to $\P$-null sets.
\end{lemma}

\begin{proof}
    Recall that we assumed the weights are i.i.d.\ in \eqref{iid}. Let $\overline\P(d\w, d\wt\w)$ be the probability measure on $\Omega^2$ that couples two copies of $\P$ as follows: $\wt\w_x = \w_x$ if $x \in \bbZ_+^2$ and $(\wt \w_x : x \notin \bbZ_+^2)$ is independent of $(\w_x : x \notin \bbZ_+^2)$.

   For 
   every $(\w,\wt\w)$, $\Geo_0^\w=\Geo_0^{\wt\w}$ and therefore $\pi^0(\w)\preceq\pi^0(\wt\w)$ or $\pi^0(\w)\succeq\pi^0(\wt\w)$. We claim that both hold almost surely.
    
    To derive a contradiction, assume that $\overline\P\{\pi^0(\w)\ne\pi^0(\wtil)\}>0$. By symmetry 
    \[\overline\P\{\pi^0(\w)\precneq\pi^0(\wt\w)\}=\overline\P\{\pi^0(\w)\succneq\pi^0(\wt\w)\}>0.\]
      Denote the two events in the above display by $A$ and $A'$. 
    By the ergodic theorem, for $\overline\P$-almost every $(\w,\wt\w)$, $T_{ke_1}(\w,\wt\w)\in A$ occurs for infinitely many integers $k$ and $T_{\ell e_1}(\w,\wt\w)\in A'$ occurs for infinitely many integers $\ell$. Thus, for $\overline\P$-almost every $(\w,\wt\w)$, there exist integers $k<\ell$ such that $T_{ke_1}(\w,\wt\w)\in A$ and $T_{\ell e_1}(\w,\wt\w)\in A'$. This implies that 
    \[\pi^{ke_1}(\w)\precneq\pi^{ke_1}(\wt\w)\coal\pi^{\ell e_1}(\wt\w)\precneq\pi^{\ell e_1}(\w).
    \]
   These inequalities prevent the coalescence
      $\pi^{ke_1}(\w)\coal\pi^{\ell e_1}(\w)$, thereby contradicting the assumption that $\{\pi^u:u\in\Z^2\}$ is a coalescing system of semi-infinite geodesics. Thus $\pi^0(\w) = \pi^0(\wt\w)$ almost surely. Standard measure theory (e.g.~ Lemma A.2 in \cite{Kur-07})
       implies   the existance of a $\sigma(\w_x : x \in \bbZ_+^2)$-measurable function $F:\Omega\to\bbX_0$ such that $\pi^0(\w) = F(\w)$ $\bbP$-almost surely.
\end{proof}

Recall definition \eqref{B-gen4} of the Busemann function $\Busgeo_\pi$. For  $\w$ in the full $\P$-probability event from Theorem \ref{thm:g400} on which Busemann functions are well-defined and for  $s\in(0,1]$, define 
	\begin{align}\label{B-def}
	\Buss^s(\w,x,y)=\Busgeo_{\cdfpi^{\Bhat,0,s}(\w)}(\w,x,y).
	\end{align}

\begin{lemma}\label{lm:B-cK}
Fix $\Bhat\in\cKhat$. 
Suppose $s\in(0,1]$ satisfies 
	\begin{align}\label{s-choice}
	\P\bigl( \forall u,v\in\Z^2:\cdfpi^{\Bhat,u,s}\coal \cdfpi^{\Bhat,v,s}\tspa\bigr) =1.
	\end{align}
Then $\Buss^s$, defined by \eqref{B-def}, is a shift-covariant recovering cocycle. For any $u\in\Z^2$ and $\{x,y\}\subset u+\Z^2_+$, $\Buss^s(x,y)$ is $\sFfor_u$-measurable, up to events of zero $\bbP$-measure.
\end{lemma}

\begin{proof}
That $\Buss^s$ is a recovering cocycle comes from Theorem \ref{thm:g400}. The shift-covariance \eqref{x-cov} and the coalescence \eqref{s-choice} allow to apply Lemma \ref{lm:pi-forward} to the system $\{\cdfpi^{\Bhat,u,s}:u\in\Z^2\}$ to conclude that  $\cdfpi^{\Bhat,u,s}$ is $\kSfor_u$-measurable, for all $u\in\Z^2$.  
Take $u\in\Z^2$ and $x,y\in u+\Z^2_+$. Then $\Lpp_{x,\cdfpi_n^{\Bhat,u,s}}$ and $\Lpp_{y,\cdfpi_n^{\Bhat,u,s}}$ are both $\kSfor_u$-measurable (for $n\ge u\cdot(e_1+e_2)$) and, consequently, so is $\Busgeo_{\cdfpi^{\Bhat,u,s}(\w)}(\w,x,y)$.
Lemma \ref{lm:B-order} and the coalescence  $\cdfpi^{\Bhat,u,s}\coal\cdfpi^{\Bhat,0,s}$ implies
    \begin{align}\label{B_u}
    \Buss^s(\w,x,y)=\Busgeo_{\cdfpi^{\Bhat,u,s}(\w)}(\w,x,y),
    \end{align}
for all $x,y,u\in\Z^2$.
Thus, we see that $\Buss^s(\w,x,y)$ is $\sFfor_u$-measurable, for all $x,y\in u+\Z^2_+$.

\eqref{B_u} also implies the shift-covariance:
	\begin{align*}
	\Buss^s(\w,x+z,y+z)
	&=\Busgeo_{\cdfpi^{\Bhat,z,s}(\w)}(\w,x+z,y+z)=\Busgeo_{z+\cdfpi^{\Bhat,0,s}(T_z\w)}(\w,x+z,y+z)\\
	&=\Busgeo_{\cdfpi^{\Bhat,0,s}(T_z\w)}(T_z\w,x,y)=\Buss^s(T_z\w,x,y),
	\end{align*}
where the first equality used \eqref{B_u} with $u=z$, the second equality used the shift-covariance \eqref{x-cov}, and the third equality used the shift-covariance \eqref{B-cov}. 
\end{proof}

Lemma \ref{lm:g428} gives the following: 

\begin{lemma}\label{lm:xhat->Bhat}
Take $\Bhat\in\cKhatcoal$. Then for $\Phat$-almost every $\what$ and all $u,x,y\in\Z^2$
	\be\begin{aligned}\label{xhat->Bhat}
	\Busgeo_{\cgeod{}^{\Bhat,u}(\what)}(\w(\what),x,y)=\Bhat(\what,x,y).
	\end{aligned}\qedhere\ee
\end{lemma}


\begin{lemma}\label{lm:Bs-Bhat}
Fix $\Bhat\in\cKhatcoal$. 
The distribution of $(\w(\what),\Bhat(\what))$ under $\Phat(d\what)$ is the same as the distribution of $(\w,\Buss^s(\w))$ under $\P(d\w)\otimes\leb(ds)$.
\end{lemma}

\begin{proof}
  Apply Lemma \ref{lm:Fubini-new} to the measurable mapping $(\w,\gamma)\mapsto(\w,\Busgeo_\gamma(\w))$ to get that the distribution of $(\w,\Busgeo_{\cdfpi^{\Bhat,u,s}}(\w))$ under $\P(d\w)\otimes\leb(ds)$ is the same as that of $(\w(\what),\Busgeo_{\cgeod{}^{\Bhat,u}(\what)}(\w(\what)))$ under $\Phat(d\what)$. The claim then follows from \eqref{B_u} and \eqref{xhat->Bhat}.
\end{proof}

Given $\Bhat\in\cKhat$, recall the random vector $\hhB(\Bhat, \what)$ defined in \eqref{h-def}. Note that if $\Bus\in\cK$, from the assumption that $\P$ is i.i.d.,
$\hhB(\Bus,\w)$ is a deterministic vector, i.e.\ $\hhB(\Bus,\w)=\E[\hhB(\Bus)]$, $\P$-almost surely.

\begin{lemma}\label{lm:orig_space}
Fix $\Bhat\in\cKhatcoal$. 
Let $h=\Ehat[\hhB(\Bhat)]$. Assume that $\Phat\{\what: \hhB(\Bhat, \what)=h\}=1$. Then for 
every $s\in(0,1]$,  $\Buss^s\in\cKcoalfor$ and $\hhB(\Buss^s)=h$, 
$\P$-almost surely.
Furthermore, $\P\{\forall s,t\in(0,1]:\Buss^s= \Buss^t\}=1$.
\end{lemma}

\begin{proof}
By Lemmas \ref{lm:coal}, \ref{lm:B-cK}, and \ref{lm:Bs-Bhat}, there exists a Borel set $D_1\subset(0,1]$ such that $\leb(D_1)=1$ and for each $s\in D_1$,
$\Buss^s$ is a 
forward measurable  shift-covariant recovering cocycle with coalescing geodesics.

The assumption $\hhB(\Bhat, \what)=h$, identity \eqref{xhat->Bhat}, and the cocycle shape theorem \eqref{B-shape}
give for $\Phat$-almost every $\what$
	\[\varlimsup_{\abs{x}_1\to\infty}\frac{\abs{\Busgeo_{\cgeod{}^{\Bhat,0}(\what)}(\w(\what),0,x)+h\cdot x}}{\abs{x}_1}=0.\]
Then Lemma \ref{lm:Fubini-new} says that there exists a Borel set $D_2\subset(0,1]$ with $\leb(D_2)=1$ and such that for each $s\in D_2$,   
	\[\P\Bigl\{\,\varlimsup_{\abs{x}_1\to\infty}\frac{\abs{\Busgeo_{\cdfpi^{\Bhat,0,s}(\w)}(\w,0,x)+h\cdot x}}{\abs{x}_1}=0\Bigr\}=1.\]
Thus, for each $s\in D_1\cap D_2$, $\Buss^s$ is a forward-measurable  shift-covariant  recovering cocycle (on $\Omega$) that satisfies
	\[\varlimsup_{\abs{x}_1\to\infty}\frac{\abs{\Buss^s(0,x)+h\cdot x}}{\abs{x}_1}=0,\quad\P\text{-almost surely}.\]
Then $n^{-1}\Buss^s(0,ne_1)\to h\cdot e_1$ and $n^{-1}\Buss^s(0,ne_2)\to h\cdot e_2$. The shift-covariance, the cocycle property, the inequalities $\Buss^s(\w,x,x+e_i)\ge\w_x\in L^{1}(\Phat)$, and Birkhoff's ergodic theorem give integrability and  $\hhB(\Buss^s)=h$, so in particular  $\Buss^s\in\cKcoalfor$. 

So far, we proved that for $\leb$-almost 
every $s\in(0,1]$,  $\Buss^s\in\cKcoalfor$ and $\hhB(\Buss^s)=h$. This implies that there exists a countable dense set $D_3\subset(0,1]$ such that for every $s\in D_3$, $\Buss^s\in\cKcoalfor$ and $\hhB(\Buss^s)=h$. 

The monotonicity of $\cdfpi^{\Bhat,0,s}$ in $s$ implies the monotonicity of $\Buss^s$ by Lemma \ref{lm:B-order}.
This monotonicity and the equal expectations imply that $\P$-almost surely, for any $s,t\in D_3$, $\Buss^s=\Buss^t$.
Using the monotonicity one more time extends this to all $s,t\in(0,1]$. 
%
\end{proof}

From Lemmas \ref{lm:Bs-Bhat} and \ref{lm:orig_space} we can now establish the strong existence claimed in Theorem \ref{thm:Bhat-B}.  

\begin{proof}[Proof of Theorem \ref{thm:Bhat-B}] 
Define $\Bus(\w,x,y)$ = $\int_0^1 \Buss^s(\w, x,y) ds$, a Borel-measurable random field on $\Omega$. By Lemma \ref{lm:orig_space}, $\bbP$-almost surely, $\Buss^s = \Bus$ for all $s \in (0,1]$ and so $\Bus \in \cKfor_c$. Lemma \ref{lm:Bs-Bhat} now implies that the joint distribution of $(\w(\what),\Bhat(\what))$ under $\Phat$ is the same as that of $(\w,\Bus(\w))$ under $\bbP$. It follows that $\Phat$-almost surely, $\Bhat(\what) = \Bus(\w(\what))$. This last claim is essentially Lemma 2.2 in \cite{Kur-07}, but we include the proof. We can uniquely (up to sets of $\bbP$-measure zero) factorize the joint distribution of $(\w(\what), \Bhat(\what))$ under $\Phat$ as the distribution $\bbP(d\w)$ of $\w$ together with a transition kernel $\eta(db\viiva\w)$ that represents the conditional distribution of $\Bhat$ given $\w$. Do the same on the other side of the equality in distribution to see that $\eta(db\viiva\w) = \delta_{\Bus(\w)}(db)$ $\bbP$-almost surely. Thus $\Phat\{\what : \Bhat(\what) = \Bus(\w(\what))\} = \Ehat[\Phat(\Bhat=\Bus\viiva\kS)]=1.$
\end{proof}

\subsection{Strong uniqueness}
Recall the order relation on cocycles: for $B,B'\in\R^{\Z^2}$, $B\preceq B'$ means $B(x,x+e_1)\ge B'(x,x+e_1)$ and $B(x,x+e_2)\le B'(x,x+e_2)$ for all $x\in\Z^2$.

\begin{lemma}\label{lm:monotonicity}
    Let $\Bus_1,\Bus_2\in\cKcoal$. Then exactly one of the following three happens: $\P(\Bus_1\precneq\Bus_2)=1$, $\P(\Bus_2\precneq\Bus_1)=1$, or $\P(\Bus_1=\Bus_2)=1$.
\end{lemma}

\begin{proof}
   On the full $\P$-probability event where  $\cgeod{}^{\Bus_1,u}\coal \cgeod{}^{\Bus_1,v}$ and $\cgeod{}^{\Bus_2,u}\coal\cgeod{}^{\Bus_2,v}$ for all $u,v\in\Z^2$, we have that $\cgeod{}^{\Bus_1,0}\preceq \cgeod{}^{\Bus_2,0}$ implies $\cgeod{}^{\Bus_1,u}\preceq\cgeod{}^{\Bus_2,u}$ for all $u\in\Z^2$. Thus   the event $\{\cgeod{}^{\Bus_1,0}\preceq \cgeod{}^{\Bus_2,0}\}$ is shift-invariant and thereby has  $\P$-probability of $0$ or $1$ by the ergodicity of $\P$. The same holds for the event $\{\cgeod{}^{\Bus_1,0}\succeq \cgeod{}^{\Bus_2,0}\}$.

    Since the $\preceq$ is a total order on $\Geo_0^\w$, we have that $\P$-almost surely, either $\cgeod{}^{\Bus_1,0}\preceq\cgeod{}^{\Bus_2,0}$ or $\cgeod{}^{\Bus_1,0}\succeq \cgeod{}^{\Bus_2,0}$. Since we just showed that these two events are trivial, we get that either $\cgeod{}^{\Bus_1,0}\preceq\cgeod{}^{\Bus_2,0}$, $\P$-almost surely, or $\cgeod{}^{\Bus_1,0}\succeq\cgeod{}^{\Bus_2,0}$, $\P$-almost surely. Lemmas \ref{lm:xhat->Bhat} and \ref{lm:g428} imply then that either $\Bus_1\preceq\Bus_2$, $\P$-almost surely, or $\Bus_1\succeq\Bus_2$, $\P$-almost surely. 

    To conclude, observe that the shift-covariance of $\Bus_1$ and $\Bus_2$ implies that the event $\{\forall x,y\in\Z^2:\Bus_1(\w,x,y)=\Bus_2(\w,x,y)\}$ is shift-invariant. Therefore, $\P(\Bus_1=\Bus_2)\in\{0,1\}$.
\end{proof}

Strong uniqueness follows.

\begin{proof}[Proof of Theorem \ref{thm:uniqueness}]
By the total order in Lemma \ref{lm:monotonicity}, there exists $i\in\{1,2\}$ such that $\P(\Bus_i\preceq\Bus_{3-i})=1$. Ergodicity of $\P$, together with the definition \eqref{h-def}, then implies $\hhB(\Bus_i)\preceq\hhB(\Bus_{3-i})$. If $\hhB(\Bus_1)=\hhB(\Bus_2)$, the the means coincide, and the ordering forces $\Bus_1=\Bus_2$ almost surely. The claim of the theorem follows from the trichotomy in Lemma \ref{lm:monotonicity}.
\end{proof}

\begin{proof}[Proof of Theorem \ref{thm:decomp}]
First, consider $\Phat$ which is ergodic under $\That$. Then $\Phat$-almost surely, $\hhB(\Bhat)$ defined by \eqref{h-def} is deterministic. Denote the value by $h$. By Theorems \ref{thm:Bhat-B} and \ref{thm:uniqueness}, 
$h\in\cH$ and $\Bhat(\what)=B^h(\w(\what))$, $\Phat$-almost surely. By Theorem \ref{thm:stexist}\eqref{h-=h+}, $B^{h-}(\w)=B^{h+}(\w)$, $\P$-almost surely, and hence $B^{h-}(\w(\what))=B^{h+}(\w(\what))$, $\Phat$-almost surely. We have thus shown that 
\[\Phat\Bigl\{\what:B^{\hhB(\Bhat,\what)-}(\w(\what),x,y)=B^{\hhB(\Bhat,\what)+}(\w(\what),x,y)=\Bhat(\what,x,y)\quad\text{for all $x,y\in\Z^2$}\Bigr\}=1.\]
Since the above event has full probability under every ergodic probability measure, the ergodic decomposition theorem \cite[Appendix B]{Ras-Sep-15-ldp} tells us that it also has full probability under any $\That$-invariant probability measure. 
\end{proof}

The following is another immediate corollary of Theorem \ref{thm:Bhat-B}.

\begin{corollary}\label{cor:coal=for}
   $\cKcoal=\cKfor$ and hence these spaces are both the same as their intersection $\cKcoalfor$.
\end{corollary}

\begin{proof} By an adaptation of the Licea-Newman \cite{Lic-New-96} coalescence argument given in
     Theorem A.1 in \cite{Geo-Ras-Sep-14}, $\cKfor\subset\cKcoal$. Forward measurability gives the finite energy condition used in the coalescence proof.  The previous inclusion then gives  $\cKfor\subset\cKcoalfor$, so we have $\cKcoalfor=\cKfor$.  Theorem \ref{thm:Bhat-B} 
      implies that also $\cKcoal\subset\cKcoalfor$ and hence $\cKcoal=\cKfor=\cKcoalfor$.
\end{proof}

Recall Theorem \ref{thm:exist} and the consequence that $\sH$ contains $\{-\nabla g(\xi\sigg): \xi \in \Uset, \sigg \in \{+,-\}\}$. By Theorem \ref{thm:uniqueness},  for each $h\in\cH$ there exists a unique $\Bush^h\in\cKcoalfor$ such that $\hhB(\Bush^h)=h$. 
%
The following is a direct consequence of Theorem \ref{thm:uniqueness}.

\begin{lemma}\label{lm:order}
     For any $h,h'\in\cH$,  either we have $h\preceq h'$ and $\P(\Bush^h\preceq\Bush^{h'})=1$ or we have $h'\preceq h$ and $\P(\Bush^{h'}\preceq\Bush^h)=1$. In particular, $\preceq$ is a total order on $\cH$ and $\cH\ni h\mapsto\Bush^h$ is nondecreasing.
\end{lemma}

Let $\cHdense$ be a countable dense subset of $\cH$. Using the monotonicity in Lemma \ref{lm:order} and the cocycle property \eqref{cocycle} that $\Bush^h$, $h\in\cHdense$, satisfy, 
define the process
\begin{align*}
    &\Bushtmp^{h-}(x,y)=\lim_{\cHdense\tsp\ni\tsp h'\nearrow h}\Bush^{h'}(x,y)\quad\text{and}\quad
    \Bushtmp^{h+}(x,y)=\lim_{\cHdense\tsp\ni\tsp h'\searrow h}\Bush^{h'}(x,y),
\end{align*}
for $x,y\in\Z^2$ and $h\in\cH$. Then for $\P$-almost every $\w$, for any $h\in\cH$ and $\sigg\in\{-,+\}$, $\Bushtmp^{h\sig}$ is a  recovering cocycle.
The following lemma says that for a fixed $h\in\cH$, the above definitions recover $\Bush^h$. 

\begin{lemma}\label{lem:extend=}
    Fix $h\in\cH$. Then $\P$-almost surely, for any $x,y\in\Z^2$, $\Bushtmp^{h-}(x,y)=\Bushtmp^{h+}(x,y)=\Bush^h(x,y)$. In particular, this holds for $\P$-almost every $\w$, simultaneously for all $h\in\cHdense$. 
\end{lemma}

\begin{proof}
    We have that $\Bushtmp^{h-}\in\cK$, $\Bushtmp^{h-}\preceq\Bush^h$, and by monotone convergence, $\hhB(\Bushtmp^{h-})=\hhB(\Bush^h)$. This implies that $\Bushtmp^{h-}=\Bush^h$, $\P$-almost surely. The case of $\Bushtmp^{h+}$ is similar.
\end{proof}
\begin{proof}[Proof of Theorem \ref{thm:stexist}]
In view of the above lemma, we will drop the overline from $\Bushtmp^{h\sig}$ and just write $\Bush^{h\sig}$.  
Furthermore, when $\Bush^{h-}=\Bush^{h+}$, we drop the sign distinction and write $\Bush^h$. In particular, for each $h\in\cH$, $\P$-almost surely, $\Bush^{h-}=\Bush^{h+}=\Bush^h$.  The claimed monotonicity follows from Lemma \ref{lm:order}. Using dominated convergence, this implies that the mean $-h$ condition holds. The cocycle, recovery, and covariance properties are closed under limits. By almost sure left- and right- continuity, uniqueness for fixed $h$ implies uniqueness of the process.
\end{proof}

We close this section with the observation that Theorem \ref{thm:Bhat-B} 
and Corollary \ref{cor:coal=for} imply that
\begin{align}
\cH=\{\hhB(\Bus):\Bus\in\cKcoal\}
=\{\hhB(\Bus):\Bus\in\cKfor\}
=\bigl\{h\in\R^2:\Bhat\in\cKhatcoal,\ \Phat\{\what: \hhB(\Bhat,\what)=h\}=1\bigr\}.
\end{align}


\section{Shift-covariant cocycles and coalescing random geodesics}\label{sec:equiv}
In this section, we work toward to the proof of Theorem \ref{thm:xB}. We begin with the proof of Lemma \ref{Api-cov}.

\begin{proof}[Proof of Lemma \ref{Api-cov}]
If $\pihat^\aabullet\in\Gcchat$, then using \eqref{B-cov}  in the second equality and the coalescence $\pihat^0\coal\pihat^z$ with Lemma \ref{lm:B-order} in the last equality,
\begin{align*}
\Busgeo_{\pihat^0(\That_z\what)}(\w(\That_z\what),x,y)
&=\Busgeo_{\pihat^0(\That_z\what)}(T_z\w(\what),x,y)
=\Busgeo_{z+\pihat^0(\That_z\what)}(\w(\what),x+z,y+z)\\
&=\Busgeo_{\pihat^z(\what)}(\w(\what),x+z,y+z)
=\Busgeo_{\pihat^0(\what)}(\w(\what),x+z,y+z).\qedhere
\end{align*}
\end{proof}

The next lemma is a useful  tool. It states that under certain  geometric constraints,  a cocycle cannot generate two distinct covariant systems of geodesics.

\begin{lemma}\label{pi=rho}
    Let $\pihat^\aabullet$ and $\gammahat^\aabullet$ be two shift-covariant non-crossing systems of geodesics such that $\Phat\{\pihat^0\preceq\gammahat^0\}=1$. Assume that either $\pihat^\aabullet\in\Gcchat$ or $\gammahat^\aabullet\in\Gcchat$. Assume also that there exists a cocycle $\Bhat\in\cKhat$ such that both $\pihat^u$ and $\gammahat^u$ are $\Bhat$-geodesics for all $u\in\Z^2$. 
    That is, 
    \begin{align}\label{pi-rule}
    \Phat\bigl\{\what:\forall u\in\Z^2,\forall n\ge u\cdot(e_1+e_2):\Bhat(\what,\pihat_n^u(\what),\pihat_{n+1}^u(\what))=\w_{\pihat_n^u(\what)}(\what)\bigr\}=1
    \end{align}
    and
    \begin{align}\label{rho-rule}
    \Phat\bigl\{\what:\forall u\in\Z^2,\forall n\ge u\cdot(e_1+e_2):\Bhat(\what,\gammahat_n^u(\what),\gammahat_{n+1}^u(\what))=\w_{\gammahat_n^u(\what)}(\what)\bigr\}=1.
    \end{align}
    Then $\Phat\{\forall u\in\Z^2:\pihat^u=\gammahat^u\}=1$.
\end{lemma}

\begin{proof}
    The two cases are proved similarly. We work out the case $\pihat^\aabullet\in\Gcchat$. By the shift-covariance \eqref{pihat56} of both $\pihat^u$ and $\gammahat^u$, it is enough to prove $\Phat\{\pihat^0=\gammahat^0\}=1$.
    We show that  $\cE_\infty=\{\what:\pihat^0(\what)\precneq \gammahat^0(\what)\}$ is a zero $\Phat$-probability event.  To arrive at a contradiction, suppose $\Phat(\cE_\infty)>0$.  
   
    The events 
        $\cE_m=\{\pihat^0_{0:m}\precneq \gammahat^0_{0:m}\}$ increase up to $\cE_\infty$ as  $m\nearrow\infty$.  
    Hence we can pick  $m\in\N$ so that $\Phat(\cE_m)>0$.   By Poincar\'e recurrence, $\P$-almost surely $\cE_m\subset\bigcup_{\ell\ge N} \That_{-\ell e_1}\cE_m$  for each $N\in\N$ \cite[Section 1.3]{Kre-85}. Then 
           \be\label{ck8321} \begin{aligned}
\Phat\Bigl(\;\bigcup_{k\ge1}\bigcup_{\ell\ge k+m}\That_{-ke_1}\cE_m\cap \That_{-\ell e_1}\cE_m\Bigr)
&= \Phat\Bigl(\;\bigcup_{k\ge1}\That_{-ke_1} \Bigl[ \cE_m\cap\bigcup_{j\ge m} \That_{-j e_1}\cE_m\Bigr] \Bigr)   \\
&= \Phat\Bigl(\;\bigcup_{k\ge1}\That_{-ke_1}\cE_m\Bigr)>0.   
        \end{aligned}\ee
      
    Let $\Omhat_0$ be the full $\Phat$-probability event on which the geodesics $\{\pihat^u(\what)\}_{u\tsp\in\tsp\Z^2}$ coalesce, $\{\gammahat^u(\what)\}_{u\tsp\in\tsp\Z^2}$ are non-crossing, $\pihat^0(\what)\preceq\gammahat^0(\what)$, and both events in \eqref{pi-rule} and \eqref{rho-rule} hold.
    The proof is concluded by showing that 
 \be\label{ck-goal2} 
 \Phat\bigl(\Omhat_0\cap \That_{-ke_1}\cE_m\cap \That_{-\ell e_1}\cE_m \bigr)  =0 \quad \forall\, k\ge 1, \;\ell\ge k+m. 
 \ee 
 Since $\Phat(\Omhat_0)=1$, this contradicts \eqref{ck8321}, which in turn forces $\Phat(\cE_\infty)=0$. 
    
    Let $\what\in \Omhat_0\cap\That_{-ke_1}\cE_m\cap\That_{-\ell e_1}\cE_m$ with $\ell\ge k+m>m$. Then the non-crossing of the geodesics $\gammahat^0(\what)$, $\gammahat^{ke_1}(\what)$, and $\gammahat^{\ell e_1}(\what)$ implies $\gammahat^0(\what)\preceq\gammahat^{ke_1}(\what)\preceq\gammahat^{\ell e_1}(\what)$. This, $\pihat^0(\what)\precneq \gammahat^0(\what)$, and the coalescence of $\{\pihat^u(\what)\}_{u\tsp\in\tsp\Z^2}$, together  imply 
\be\label{ck8341}  \pihat^u(\what)\precneq \gammahat^{v}(\what) \quad  \text{for } \ u,v\in\{ke_1,\ell e_1\} . \ee  
    Next, 
    $\what\in \That_{-ke_1}\cE_m\cap \That_{-\ell e_1}\cE_m$ says that  
   $\pihat^u(\what)$ and $\gammahat^u(\what)$ separate in the first $m$ steps, for both $u\in\{ke_1,\ell e_1\}$. Once separated, locally rightmost geodesics from $u$ cannot meet again, and so 
\be\label{ck9041}  \pihat_{o+m:\infty}^u(\what)\cap \gammahat_{o+m:\infty}^u(\what)=\varnothing \quad\text{for both } \ u\in\{ke_1,\ell e_1\} .  \ee

   We draw the conclusions from the observations above.  
 Since $\pihat^{\ell e_1}(\what)$ starts strictly to the right of $\gammahat^{ke_1}(\what)$ but by \eqref{ck8341} ends up strictly to its left, $\gammahat^{ke_1}(\what)$ and $\pihat^{\ell e_1}(\what)$ must intersect.   
    Denote their first intersection point by $z=\gammahat_n^{ke_1}(\what)=\pihat_n^{\ell e_1}(\what)$.  Since $z\in\pihat^{\ell e_1}(\what)$ we have $z\ge\ell e_1$. This implies 
    $n-k=(z-ke_1)\cdot(e_1+e_2)
    \ge  \ell-k\ge m$. Since $z\in\gammahat^{ke_1}(\what)$, \eqref{ck9041} implies $z\not\in\pihat^{ke_1}(\what)$.

    Let $x$ denote the coalescence point of  $\pihat^{ke_1}(\what)$ and $\pihat^{\ell e_1}(\what)$.   Since $z\in\pihat^{\ell e_1}(\what)\setminus\pihat^{ke_1}(\what)$, $z$ must lie on $\pihat^{\ell e_1}(\what)$ before $x$, and thereby $z$ lies strictly to the right of $\pihat^{ke_1}(\what)$. Thus in coordinatewise  ordering  
\[ke_1\le x \quad \text{and}\quad  \ell e_1\le z\le x.\]
 
Since $\what$ is in the event in \eqref{pi-rule} and since $x$ is on $\pihat^{ke_1}(\what)$ and $z\le x$ are both on $\pihat^{\ell e_1}(\what)$,
we have $\Lpp_{ke_1,x}(\w(\what))=\Bhat(\what,ke_1,x)$ and $\Lpp_{z,x}(\w(\what))=\Bhat(\what,z,x)$.
Similarly, since $z$ is on $\gammahat^{ke_1}(\what)$ and $\what$ is in the event in \eqref{rho-rule} we have $\Lpp_{ke_1,z}(\w(\what))=\Bhat(\what,ke_1,z)$.
By the cocycle property,  
\begin{align*}
    &\Lpp_{ke_1,x}(\w(\what))-\Lpp_{ke_1,z}(\w(\what))-\Lpp_{z,x}(\w(\what))\\
    &\qquad=\Bhat(\what,ke_1,x)-\Bhat(\what,ke_1,z)-\Bhat(\what,z,x)=0.
\end{align*}
This implies that $z$ is on some  geodesic from $ke_1$ to $x$.  But we observed above that  $z$ lies strictly to the right of $\pihat^{ke_1}(\what)$.   We have a contradiction because  by definition  $\pihat^{ke_1}(\what)$ gives the rightmost  geodesic from $ke_1$ to $x$. 
This contradiction verifies \eqref{ck-goal2}.
\end{proof}

As an application of the above lemma, we obtain Lemma \ref{phih-unique}: 

\begin{proof}[Proof of Lemma \ref{phih-unique}]
Define
\be\label{to3200}   \pi^u(\what)=\inf\{\gamma\in\Geo_u^{\w(\what)}:\gamma\text{ is a $\Bhat$-geodesic in $\Geo_u^{\w(\what)}$}\}.\ee 
The above set is not empty because it contains $\cgeod{}^{\Bhat,u}(\what)$. This gives $\pi^u\preceq\cgeod{}^{\Bhat,u}$. The $e_1$ tiebreaker guarantees that $\cgeod{}^{\Bhat,u}$ stays weakly to the right of every   $\Bhat$-geodesic out of $u$. The uniqueness claim follows from showing that $\pi^u=\cgeod{}^{\Bhat,u}$.


We first prove that $\Phat$-almost surely     $\pi^u(\what)\in\Geo_u^{\w(\what)}$ and  $\pi^u(\what)$ is a $\Bhat$-geodesic. 
This is clear if the infimum is attained. When the infimum is not attained, there exists a sequence $\gamma^k\in\Geo_u^\w$ of $\Bhat$-geodesics in $\Geo_u^\w$ that is strictly decreasing to $\pi^u$. Then $\pi^u\in\Geo_u^\w$. 
To see it is a $\Bhat$-geodesic, take any two consecutive points $\pi_i^u$ and $\pi_{i+1}^u$. Take $k$ large enough so that $\gamma^k_{i:i+1}=\pi_{i:i+1}^u$. Then, since $\gamma^k$ is a $\Bhat$-geodesic,
\[\Bhat(\pi_i^u,\pi_{i+1}^u)=\Bhat(\gamma^k_i,\gamma^k_{i+1})=\w_{\gamma^k_i}=\w_{\pi_i^u}.\]

The shift-covariances \eqref{w-cov} of the weights $\what\mapsto\w(\what)$, the shift-covariance \eqref{Geo55} of the geodesic trees, and the shift-covariance \eqref{covariance} of $\Bhat$ imply that $\pi^\aabullet$ is a shift-covariant system of random geodesics. We argue that the system is non-crossing. For this, suppose $\pi^v$ and $\pi^u$ meet and then eventually separate at a point $z$. For concreteness, suppose $\pi^v\aprecneq\pi^u$. Take $n$ large enough so that $\pi^v_n\precneq\pi^u_n$. Consider the rightmost geodesic $\sigma^{u,\pi^v_n}(\w)$ from $u$ to $\pi^v_n$. The concatenation of $\pi^u_{u:z}$ and $\pi^v_{z:\infty}$ is a $\Bhat$-geodesic and, therefore, the concatenation of $\pi^u_{u:z}$ and $\pi^v_{z:n}$ is a geodesic in environment $\w$. It then must be weakly to the left of $\sigma^{u,\pi^v_n}$.  In particular, $z\preceq \sigma^{u,\pi^v_n}$. 

We prove by contradiction that $z\in \sigma^{u,\pi^v_n}$. 
If not, then $z$ is strictly to the left of $\sigma^{u,\pi^v_n}$.  Now $\sigma^{u,\pi^v_n}$ starts from $u$, goes  strictly to the right of $\pi^u$  (at the level of $z$), and then   strictly to the  left of $\pi^u$  (at level $n$).  This violates the locally-rightmost property of  $\pi^u$. Thus   $z\in\sigma^{u,\pi^v_n}$. 

The inclusion $z\in\sigma^{u,\pi^v_n}$ forces $\sigma^{u,\pi^v_n}$ to be  the concatenation of $\pi^u_{u:z}$ and $\pi^v_{z:n}$. Since this holds for all large enough $n$,   the concatenation of $\pi^u_{u:z}$ and $\pi^v_{z:\infty}$
is a locally rightmost semi-infinite geodesic and hence
in the tree $\Geo_u^\w$. This provides a $\Bhat$-geodesic in $\Geo_u^\w$ strictly to the left of $\pi^u$, violating the minimality of $\pi^u$ in \eqref{to3200}. This contradiction implies that if $\pi^u$ and $\pi^v$ ever meet, they have to coalesce.

Now we have a shift-covariant non-crossing system of geodesics $\pi^\abullet$, a shift-covariant coalescing system of geodesics $\cgeod{}^{\Bhat,\aabullet}$, both are  systems are $\Bhat$-geodesics, and $\pi^u\preceq\cgeod{}^{B,u}$  $\Phat$-almost surely and for all $u\in\Z^2$. By Lemma \ref{pi=rho},  $\pi^\aabullet=\cgeod{}^{\Bhat,\aabullet}$. Thus there are no other  $\Bhat$-geodesics in   $\Geo^{\w(\what)}$ except those from  the system $\cgeod{}^{\Bhat,\aabullet}$.    
\end{proof}

The next two lemmas give the equivalence between the cocycle-focused approach of this paper and the geodesic-focused approach of  \cite{Ahl-Hof-16-} and complete the proof of Theorem \ref{thm:xB}.

\begin{lemma}\label{pi=xB} Let $\pihat^\aabullet\in\Gcchat$ and let $\Bhat=\Busgeo_{\pihat^0}$. 
Then $\Phat\{\forall u\in\Z^2:\cgeod{}^{\Bhat,u}=\pihat^u\}=1$.\qedhere
\end{lemma}

\begin{proof}  
The system $\{\cgeod{}^{\Bhat,u}\}$ generated by $\Bhat$ is noncrossing because if $\cgeod{}^{\Bhat,u}$ and $\cgeod{}^{\Bhat,v}$ ever intersect, their subsequent steps   are determined by $\Bhat$ and hence identical. Furthermore, $\cgeod{}^{\Bhat,0}$ is by definition the geodesic of $\Bhat$ that takes $e_1$ steps at ties $\Bhat(x,x+e_1)=\Bhat(x,x+e_2)$, and hence $\pihat^0\preceq\cgeod{}^{\Bhat,0}$  by Lemma \ref{lm:Arho6}.  Both \eqref{pi-rule} and \eqref{rho-rule} are satisfied. The claim $\cgeod{}^{\Bhat,u}=\pihat^u$ follows from Lemma \ref{pi=rho}.
\end{proof}

\begin{lemma}\label{pi=xB.b} Suppose $(\Omhat, \kShat,\Phat,\That)$ is ergodic. Then   $\pihat^\aabullet\in\Gcchat$ implies $\Busgeo_{\pihat^0}\in\cKhatcoal$. \end{lemma}
\begin{proof}
Theorem \ref{thm:g400} and Lemma \ref{Api-cov} imply that $\Busgeo_{\pihat^0}$ is a shift-covariant recovering cocycle. To prove that $\Busgeo_{\pihat^0}$ is $L^1$ we bound it with  the Busemann process.
Define, relative to the southeast order $\preceq$ on the simplex $\Uset=[e_2, e_1]$ of direction vectors, 
\[
\underline{\xi}(\what)=\varliminf_{n\to\infty} \pihat^0_n(\what)/n 
\qquad \text{ and }\qquad
\overline{\xi}(\what) = \varlimsup_{n\to\infty} \pihat^0_n(\what)/n .\]
Because $\pihat^{\aabullet}$ is shift-covariant and coalescing, $\overline{\xi}$ and $\underline{\xi}$ are shift-invariant random variables and then by the ergodicity assumption  almost surely constant. By Lemma 5.1 in \cite{Gro-Jan-Ras-25}, $e_2\precneq\underline{\xi} \preceq\overline{\xi} \precneq e_1$. This result implicitly uses Martin's asymptotic \eqref{eq:Martin}. Since $\shape$ has at most countably many nondifferentiability points, there exist points of differentiability $\zeta\precneq\eta$ in $\ri\Uset$ such that 
\begin{align}\label{tmp1159}
\zeta\precneq\underline\xi\preceq\overline\xi\precneq\eta
\end{align}
at which $\shape(\zeta)\neq \nabla \shape(\underline{\xi})\cdot \zeta$ and $\shape(\eta)\neq \nabla \shape(\overline{\xi})\cdot \eta$. The latter claim again relies on   \eqref{eq:Martin} which ensures  that no linear segment of the shape   abuts the coordinate boundaries.

By Theorem \ref{thm:stexist}, $(\Omhat,\kShat)$ supports an $\kS$-measurable copy of the tilt-indexed Busemann process. In particular, both $\Bhat^{-\nabla\shape(\zeta)}(\what) = \Bus^{-\nabla\shape(\zeta)}(\w(\what))$ and $\Bhat^{-\nabla\shape(\eta)}= \Bus^{-\nabla\shape(\eta)}(\w(\what))$ are in $\cKhatcoal$.
Abbreviate $\underline\pi=\cgeod{}^{\Bhat^{-\nabla\shape(\zeta)},0}$ and  $\overline\pi=\cgeod{}^{\Bhat^{-\nabla\shape(\eta)},0}$. By Theorem 4.3 in \cite{Geo-Ras-Sep-17-ptrf-2} (or Theorem 5.8(h) of \cite{Gro-Jan-Ras-25-}),
any limit point of $\underline\pi_n/n$ is contained in the set of directions $\{\xi \in \ri \Uset : \nabla \shape(\zeta) \in \partial \shape(\xi)\}.$
A similar statement holds for the limit points of $\overline\pi_n/n$.  This, the definitions of $\underline\xi$ and $\overline\xi$, and the inequalities \eqref{tmp1159} give that $\Phat$-almost surely,
\be\label{1359}
\cgeod{}^{\Bhat^{-\nabla\shape(\zeta)},0}\precneq\pihat^0\precneq\cgeod{}^{\Bhat^{-\nabla\shape(\eta)},0}.
\ee
Combine Lemmas \ref{lm:xhat->Bhat} and  \ref{lm:B-order} to get the  $\Phat$-almost sure bounds 
\begin{align}\label{temp1171}
\Bhat^{-\nabla\shape(\zeta)}=\Busgeo_{\cgeod{}^{\Bhat^{-\nabla\shape(\zeta)},0}}
\preceq
\Busgeo_{\pihat^0}
\preceq
\Busgeo_{\cgeod{}^{\Bhat^{-\nabla\shape(\eta)},0}}=\Bhat^{-\nabla\shape(\eta)}.
\end{align}
The integrability of the $\Bhat^{h\sig}$ cocycles (Theorem \ref{thm:exist}) implies that 
 $\Busgeo_{\pihat^0}$ is $L^1$.
%
%
%
\end{proof}   

Theorem \ref{thm:xB} comes  from the combination of Lemmas \ref{pi=xB} and \ref{pi=xB.b}. 

\medskip 

\appendix
\section{Technical lemmas}\label{app:aux}
\subsection{Measurability of the geodesic tree}
We begin with measurability of $\Geo_u^\w$.
For $u,x\in\Z^2$ with $x\ge u$ and $i\in\{1,2\}$ define
 the $\kSfor_u$-measurable random variables
	\[\edge_{u,x}^i(\w)=\begin{cases}
	1&\text{if }\forall m\ge x\cdot(e_1+e_2)+1\ \exists \sigma\in\Geo_{u,m}^\w:x,x+e_i\in\sigma,\\
	0&\text{otherwise.}
	\end{cases}\]
Think of $\edge_{u,x}^i=1$ as opening the edge $(x,x+e_i)$. 
Note that if $\edge_{u,x}^i=1$ then $\edge_{u,x+e_i}^j=1$ for some (or both) $j\in\{1,2\}$.
Starting at $u$ and following open edges gives a locally-rightmost semi-infinite geodesic
and, conversely, the edges of any locally-rightmost semi-infinite geodesic started at $u$ are all open. Hence 
$\edge_u=\bigl\{\edge_{u,x}^i:x\in u+\Z_+^2, \, i\in\{1,2\}\bigr\}$ is an $\kSfor_u$-measurable way to encode the random tree $\Geo_u^\w$.  It is not hard to see that $\Geo_u^\w$ is also a closed set in the product-discrete topology on paths. Since the space of paths rooted at $u$ is compact in this topology, this implies that $\Geo_u^\w$ is compact.

\subsection{Measurability on $\bbX_u$}
Without loss of generality, we consider $u=0$ for notational simplicity. We begin with some preliminary observations about the path space $\bbX_0$. Recall that for $\gamma,\pi\in\bbX_0$, the metric distance between $\gamma$ and $\pi$ is $d(\gamma,\pi) = \sum_{i=0}^\infty 2^{-(i+1)}\one_{\{\gamma_i\neq\pi_i\}}$.

We prove measurability of the expression in \eqref{eq:inf-formula-rational}.

\begin{lemma}\label{lem:infmeas}
If $\nu^\w$ is a regular conditional distribution on $(\Omhat,\kShat)$ given $\kS$, then the function 
\[
(\what, \rho) \in \Omhat\times\bbX_0 \mapsto \nu^{\w(\what)}(\cgeod{}^{\Bhat,0}\preceq \rho)\in[0,1]
\]
is jointly  $(\kS,\sB(\bbX_0))$-measurable.  \end{lemma}

\begin{proof}
Denote by $F(\what,\rho)$ the function in the statement. It suffices to show that $F$ is the limit of jointly measurable functions. We begin with the observation that for each $n \in \bbN$ and $\what \in \Omhat$
\begin{align*}
\rho \mapsto F_n(\what,\rho)= \nu^{\w(\what)}(\cgeod{0:n}^{\Bhat,0}\preceq \rho_{0:n})
\end{align*}
is continuous. To see this, note that if $\rho^k \to \rho$, then for all sufficiently large $k$, $\rho^k_{0:n}=\rho_{0:n}$. $\bbX_0$ is separable, being compact, so it follows from $\kS$-measurability of $\nu^{\w(\what)}$ that $F_n$ is $(\kS,\sB(\bbX_0))$-jointly measurable. See, e.g., Lemma 4.51 in \cite{Ali-Bor-06}. 

We have that
\[
\{\cgeod{}^{\Bhat,0}\preceq \rho\} = \bigcap_n \{\cgeod{0:n}^{\Bhat,0} \preceq \rho_{0:n}\}.
\]
From continuity of measure, $F(\what,\rho) = \lim_n F_n(\what,\rho)$. Therefore $F$ is measurable.
\end{proof}

\begin{corollary}\label{cor:infmeas}
    For $s \in [0,1]$, the $\bbX_0$-valued random variable $\cdfpi^{\Bhat,0,s}$ in \eqref{eq:inf-formula-rational} is $\kS$-measurable.
\end{corollary}

\begin{proof}
Fix a deterministic finite admissible path $\sigma_{0:n}$ rooted at the origin. Then the path defined by $\rho_{\sigma_{0:n}}(\w(\what)) = \sup\{\pi \in \Geo_0^{\w(\what)} : \pi_{0:n}=\sigma_{0:n}\}$ if such a $\pi$ exists and $\rho_{\sigma_{0:n}}(\w(\what))=\bbZ_+e_1$ otherwise is measurable by the results of the previous sub-section. Moreover, this collection enumerates $\RI_0^\w$.  It then follows from the previous result that 
\[
\what\mapsto\nu^{\w(\what)}\{ \what: \cgeod{}^{\Bhat,0}(\what)\preceq \rho_{\sigma_{0:n}}(\w(\what))\}\in[0,1]
\]
is $\kS$-measurable. Then for each $s \in (0,1)$ and each finite admissible path $\sigma_{0:n}$, the event
\[
\{\nu^{\w(\what)}(\cgeod{}^{\Bhat,0}\preceq \rho_{\sigma_{0:n}}(\w(\what))) \geq s\}
\]
is $\kS$-measurable. Fix some finite path $\pi_{0:m}$. The event
\be\label{aux_90}
\{\what: \cdfpi^{\Bhat,0,s}_{0:m}(\what) = \pi_{0:m}\}
\ee
is measurable because it is equal to the intersection of two events: (i) there exists a path $\sigma_{0:n}$ which is on $\Geo_0^{\w(\what)}$, $n \geq m$, where $\pi_{0:m}\subset \sigma_{0:n}$ for which 
$\nu^{\w(\what)}(\cgeod{}^{\Bhat,0}\preceq \rho_{\sigma_{0:n}}(\w(\what))) \geq s$,
\[
\bigcup_{n \geq m}\bigcup_{\substack{\sigma_{0:n} \text{ up-right}: \\
\sigma_{0:m} = \pi_{0:m}}} \{\sigma_{0:n} \text{ is on }\Geo_0^{\w(\what)}, \nu^{\w(\what)}(\cgeod{}^{\Bhat,0}\preceq \rho_{\sigma_{0:n}}(\w(\what))) \geq s\};
\]
and (ii) for any $n \geq m$ and any finite path $\sigma_{0:n}$ which lies on the tree $\Geo_0^{\w(\what)}$ with the property that $\sigma_{0:m} \precneq \pi_{0:m}$, we have
$\nu^{\w(\what)}(\cgeod{}^{\Bhat,0}\preceq \rho_{\sigma_{0:n}}(\w(\what))) < s$:
\[
\bigcap_{n \geq m}\bigcap_{\substack{\sigma_{0:n} \text{ up-right} : \\ \sigma_{0:m} \precneq \pi_{0:m}}} \big(\{\sigma_{0:n}\text{ is on }\Geo_0^{\w(\what)},\nu^{\w(\what)}(\cgeod{}^{\Bhat,0}\preceq \rho_{\sigma_{0:n}}(\w(\what))) < s \}\cup\{\sigma_{0:n} \text{ is not on }\Geo_0^{\w(\what)}\}\big).
\]
Events of the type \eqref{aux_90} generate $\sB(\bbX_0)$, so the claim follows. 
\end{proof}

We also have the following Lemma concerning generation of $\sB(\bbX_0)$.
\begin{lemma}\label{lem:gen}
 The family of events $\{\geodd{0:\infty} \in \bbX_0 : \geodd{0:\infty} \preceq \pi_{0:\infty}\}$, $\pi_{0:\infty}\in\bbX_0$ generates $\sB(\bbX_0)$.
\end{lemma}

\begin{proof}
 Cylinder events of the form $\{\geodd{0:\infty}\in\bbX_0:\geodd{0:n}=\geoddd{0:n}\}$ are intersections of events of the form $\{\geodd{0:\infty}\in\bbX_u:\geodd{n}=x\}$, $x\in \Z_+^2$. For a given $x\in \Z_+^2$, let $\geoddd{0:\infty}^x$ be the up-right path that starts with $x\cdot e_1$ $e_1$-steps, then takes $x\cdot e_2$ $e_2$-steps, getting to $x$, then from there only takes $e_1$ steps.   Then, for $x\ne ne_2$, we have 
 \[\{\geodd{0:\infty}\in\bbX_0:\geodd{n}=x\}=\{\geodd{0:\infty}\in\bbX_0:\geodd{0:\infty}\preceq \geoddd{0:\infty}^x\}\setminus\{\geodd{0:\infty}\in\bbX_0:\geodd{0:\infty}\preceq\geoddd{0:\infty}^{x+e_2-e_1}\}.\]
 For $x=ne_2$ we have
 \[\{\geodd{0:\infty}\in\bbX_0:\geodd{n}=x\}=\{\geodd{0:\infty}\in\bbX_0:\geodd{0:\infty}\preceq \geoddd{0:\infty}^x\}.\qedhere\]
\end{proof}

\subsection{Non-existence of trivial Busemann geodesics}
\begin{lemma}\label{lem:xhat nontrivial}
Assume \eqref{var>0}, \eqref{iid}, and $\Ehat[\w_0^2]<\infty$. Let $\Bhat \in \cKhat$. Then \eqref{no e2} holds.
\end{lemma}

\begin{proof}
By the shift invariance of $\Phat$ and shift covariance of $\Bhat$, it is enough to consider $u=0$. On the event  $\cgeod{}^{\Bhat,0}=\Z_+ e_2$  we have, by 
the cocycle and the recovery properties, 
\begin{align*} 
&\sum_{k=0}^{n-1} \w_{ke_2}  + \Bhat(ne_2, e_1+ne_2) 
= \sum_{k=0}^{n-1} \Bhat(ke_2, (k+1)e_2)  + \Bhat(ne_2, e_1+ne_2) \\
&\quad =  \Bhat(0,e_1) + \sum_{k=0}^{n-1} \Bhat(e_1+ke_2, e_1+(k+1)e_2)  
    \ge   \Bhat(0,e_1) + \sum_{k=0}^{n-1} \w_{e_1+ke_2}
\end{align*} 
from which 
\begin{align*} 
\frac{\Bhat(ne_2, e_1+ne_2)}{\sqrt n} 
   \ge  \frac1{\sqrt n}  \sum_{k=0}^{n-1} (\w_{e_1+ke_2} - \w_{ke_2}) + \frac{\Bhat(0, e_1)}{\sqrt n}.
\end{align*}
Now the left-hand side goes to $0$ in probability and hence almost surely along some subsequence of any given subsequence, while the limsup of the right-hand side is infinite almost surely.
\end{proof}

\subsection{Sufficient condition for ergodicity}
We next prove the sufficient condition for a cocycle to be ergodic, which was recorded as Lemma \ref{lem:ergosuf} above.
\begin{proof}[Proof of Lemma \ref{lem:ergosuf}]
Replacing $\hat u$ with $-\hat u$, if necessary, we can restrict attention to the case $\hat u\in(\R_+\times\R_-)\setminus\{0\}$.
By Lemma \ref{lem:hsupdif}\eqref{hsupdif.a} and the observation following \eqref{eq:Martin}, 
$\hhB(\Bhat)\cdot \hat u \in \bigl(-\partial\shape(\ri\Uset)\bigr)\cdot\hat u=\{ h\cdot \hat u : h \in -\partial \shape(\ri\Uset) \}$,
$\Phat$-almost surely.  
Since, by assumption, 
$\hhB(\Bhat)\cdot \hat u = \Ehat[\hhB(\Bhat)]\cdot \hat u$, $\Phat$-almost surely, 
it follows that 
$\Ehat[\hhB(\Bhat)]\cdot \hat u \in\bigl(-\partial\shape(\ri\Uset)\bigr)\cdot\hat u$.  
That is, there exists $h \in -\partial \shape(\ri\Uset)$ such that 
$h\cdot \hat u = \Ehat[\hhB(\Bhat)]\cdot \hat u$.  
The assumption then implies that 
$\hhB(\Bhat)\cdot \hat u = h\cdot \hat u$, $\Phat$-almost surely.
 This implies that $\hhB(\Bhat)=h$, with $\Phat$-probability one. Indeed, Lemma 4.6(a) in \cite{Jan-Ras-20-aop} says that on the event $\{h\cdot e_1<\hhB(\Bhat)\cdot e_1\}$, 
we have $h\cdot e_2>\hhB(\Bhat)\cdot e_2$ and, therefore, $h\cdot\hat u<\hhB(\Bhat)\cdot\hat u$. A similar argument goes on the event $\{h\cdot e_1>\hhB(\Bhat)\cdot e_1\}$. Therefore, $h\cdot e_1=\hhB(\Bhat)\cdot e_1$ holds on the full event $\{\hhB(\Bhat)\cdot\hat u=h\cdot\hat u\}$. Lemma 4.6(a) in \cite{Jan-Ras-20-aop} implies then that the two vectors are almost surely equal as claimed. In particular, $h=\Ehat[\hhB(\Bhat)]$.
\end{proof}

\section{Busemann functions generated by geodesics}  \label{a:geoBus}

First we prove that in almost every environment $\w$, each nontrivial  semi-infinite geodesic generates a recovering cocycle. Then we record  some basic properties of these cocycles.

\begin{theorem}\label{thm:g400} 
Assume weights are i.i.d.\ with $p>2$ moments.   Then there exists an event $\Omega_0$ of full probability on which the following holds. For each $\w\in\Omega_0$ and every semi-infinite geodesic  $\pi_{k:\infty}$ in the environment $\w$  such that  
$\pi_n\cdot\evec_i\to\infty$  for both $i\in\{1,2\}$,   there exists a finite Busemann function 
\be\label{g408a}  \Busgeo(\w,x,y) =  \lim_{n\to\infty}  [\Lpp_{x,\pi_n}(\w) - \Lpp_{y, \pi_n}(\w) ] \quad \forall x,y\in\Z^2    \ee
that recovers the weights $\w$: 
\[  \w_x = \Busgeo(\w,x,x+\evec_1) \wedge \Busgeo(\w,x,x+\evec_2).\qedhere \]  
\end{theorem} 
 

\begin{proof}    
Recovery at $x$  follows once the limits in \eqref{g408a} are proved for $x$ and  $y\in\{x+\evec_1, x+\evec_2\}$:   for large enough $n$, 
\begin{align*}
&\Lpp_{x, \pi_n} = \w_x + \Lpp_{x+\evec_1, \pi_n}\vee \Lpp_{x+\evec_2, \pi_n}  \\
&\qquad \implies \ \ \w_x = (\Lpp_{x, \pi_n}-\Lpp_{x+\evec_1, \pi_n})\wedge (\Lpp_{x, \pi_n}- \Lpp_{x+\evec_2, \pi_n} )  \underset{n\to\infty}\longrightarrow  \Busgeo(x,x+\evec_1) \wedge \Busgeo(x,x+\evec_2).  
\end{align*}

We begin with a purely deterministic lemma that gives the limit \eqref{g408a} in a northeast quadrant. 

\begin{lemma} \label{lm:g406}  Consider a fixed   weight configuration $\w\in\R^{\Z^2}$ and $k\in\Z$.  Suppose $\pi_{k:\infty}$ is a semi-infinite geodesic such that $\pi_n\cdot\evec_i\to\infty$ for $i\in\{1,2\}$.   Then  for all $x\in\Z^2$ the  monotone nondecreasing  limit 
\be\label{g408b} 
 \Busgeo(x,\pi_k) =  \lim_{n\to\infty}  ( \Lpp_{x,\pi_n} - \Lpp_{\pi_k, \pi_n} )    \ee
exists in $(-\infty, \infty]$.  
On the  quadrant $\pi_k+\Z_+^2$ we have a finite  Busemann function 
\be\label{g408}  
\Busgeo(x,y) =  \lim_{n\to\infty}  ( \Lpp_{x,\pi_n} - \Lpp_{y, \pi_n} ) .    \ee
\end{lemma} 

\begin{proof} 
  Given $x$, let   $N$ be any index such that $\pi_{N}\ge x$. 
Then for $n\ge N$, 
\begin{align*}
\Lpp_{x,\pi_{n+1}} - \Lpp_{\pi_{N}, \pi_{n+1}} &\ge (\Lpp_{x, \pi_{n}} + \Lpp_{\pi_{n}, \pi_{n+1}} ) - (\Lpp_{\pi_{N}, \pi_{n}} + \Lpp_{\pi_{n}, \pi_{n+1}} ) 
= \Lpp_{x, \pi_{n}}  - \Lpp_{\pi_{N}, \pi_{n}} .  
\end{align*} 
Thus  this monotone nondecreasing  limit exists: 
\be\label{g450} 
\lim_{n\to\infty}  ( \Lpp_{x,\pi_n} - \Lpp_{\pi_{N}, \pi_n} )   \in [ \Lpp_{x, \pi_{N}}, \infty] . 
\ee
Furthermore, since $\Lpp_{\pi_k, \pi_n}=\Lpp_{\pi_k, \pi_N}+\Lpp_{\pi_N, \pi_n}$ for $k\le N\le n$,  we have  this monotone nondecreasing  limit: 
\be\label{g452} 
\Busgeo(x,\pi_k)=\lim_{n\to\infty}  ( \Lpp_{x,\pi_n} - \Lpp_{\pi_k, \pi_n} )   \in [ \Lpp_{x, \pi_{N}}-\Lpp_{\pi_k, \pi_N}, \infty] . 
\ee
Now let  $x\in \pi_k+\Z_+^2$.     
For $n$ such that  $ \pi_k\le x\le \pi_n$, 
\[  \Lpp_{\pi_k, \pi_n}\ge  \Lpp_{\pi_k, x} + \Lpp_{x,\pi_n}  
\ \implies \ \Lpp_{x,\pi_n}  -  \Lpp_{\pi_k, \pi_n}  \le -\Lpp_{\pi_k, x}. \]  
Thus for any $x\ge \pi_k$ we have the finite limit 
\be\label{g454} 
\Busgeo(x,\pi_k)=\lim_{n\to\infty}  ( \Lpp_{x,\pi_n} - \Lpp_{\pi_k, \pi_n} )   \in [ \Lpp_{x, \pi_{N}}-\Lpp_{\pi_k, \pi_N}, \, -\Lpp_{\pi_k, x} ]   
\qquad \forall N \text{ such that } \pi_N\ge x.  
\ee
This defines a finite Busemann function 
\be\label{g458}  
\Busgeo(x,y) = \Busgeo(x,\pi_k) - \Busgeo(y,\pi_k)  =  \lim_{n\to\infty}  ( \Busgeo_{x,\pi_n} - \Busgeo_{y, \pi_n} )     \ee
in the quadrant $\pi_k+\Z_+^2$. 
\end{proof} 

Returning to the  proof of  Theorem \ref{thm:g400}, it remains to verify that when weights are i.i.d.\ with $p>2$ moments, we can construct a full-probability event $\Omega_0$ on which  $\Busgeo(x,\pi_k)$ in  \eqref{g408b} is finite for each $x\in\Z^2$ and for every semi-infinite geodesic $\pi_{k:\infty}$ such that $\pi_n\cdot\evec_i\to\infty$  for both $i\in\{1,2\}$.    This can be achieved by combining   known properties of geodesics and Busemann functions in the corner growth model.   Namely, there exists a  full-probability event $\Omega_0$ on which the following properties hold. 
\begin{enumerate} [label={\rm(\alph*)}, ref={\rm\alph*}] \itemsep=2pt 
\item   Each semi-infinite geodesic $\pi_{k:\infty}$  is directed into some $\cU_\xi$ for some $\xi\in[\evec_2, \evec_1]$, meaning that,  as $n\to\infty$,  all the limit points of $\pi_n/n$ lie in $\Uset_\xi$. See Theorem 2.1(i) in \cite{Geo-Ras-Sep-17-ptrf-2} or Theorem 3.14(e) in \cite{Gro-Jan-Ras-25-}. The former assumes the weights are bounded below, but the proof does not require this assumption. 
\item   The only geodesics directed towards $\evec_i$ are the trivial ones of the form $x+\Z_+\evec_i$  (Lemma 5.1 in \cite{Gro-Jan-Ras-25}). 
\item For any sequence $\{u_n\}\subset\Z^2$ such that,  as $n\to\infty$,  $u_n\cdot\evec_i\to\infty$  for both $i\in\{1,2\}$ and  the set of  limit points of $\{u_n/n\}$ is bounded away from $\{\evec_2, \evec_1\}$,  
\be\label{g602}   \varlimsup_{n\to\infty}\abs{\Lpp_{x,u_n}-\Lpp_{y,u_n}} <\infty  \quad\text{$\P$-almost surely} \quad \forall x,y\in\Z^2. 
\ee
This comes from the zero-temperature version of Theorem 4.14 in \cite{Jan-Ras-20-aop}, or by taking the intersection of the full probability events in Theorem 6.1 of \cite{Geo-Ras-Sep-17-ptrf-1} over a countable dense collection of exposed points and  maximal linear segments in $\,]\evec_2, \evec_1[\,$.  
\end{enumerate} 
 
Since we know from Lemma \ref{lm:g406} that  $\Busgeo(x,\pi_k)>-\infty$  for all $x$ and  $\Busgeo(x,\pi_k)<\infty$  for $x$ in a northeast quadrant, it is enough to prove the following statement on the event $\Omega_0$: 
\be\label{g609} 
\text{if $\Busgeo(x,\pi_k)<\infty$, then  $\Busgeo(x-\evec_1,\pi_k)<\infty$ and $\Busgeo(x-\evec_2,\pi_k)<\infty$. } 
\ee
We prove the case $\Busgeo(x-\evec_1,\pi_k)<\infty$, the other one being entirely analogous.  

We can now assume that  for some $\xi\in\,]\evec_2, \evec_1[\,$, as $n\to\infty$,  all the limit points of $\pi_n/n$ lie in $\Uset_\xi$.   $\Uset_\xi$ is a compact segment (possibly a singleton) contained  in the open segment $\,]\evec_2, \evec_1[\,$.   By the curvature of the shape function close to the extreme direction $\evec_2$ implied by \eqref{eq:Martin},  we  can  pick a direction $\zeta\in\,]\evec_2, \xi[\,$ so that the segment $\Uset_\zeta$ lies strictly to the northwest of the segment  $\Uset_\xi$. 
As in equation (2.12) in  \cite{Jan-Ras-Sep-23} or in Section 4 of \cite{Geo-Ras-Sep-17-ptrf-2},   the Busemann function  $B^{\zeta+}$  defines  the Busemann geodesic  $\cgeod{}^{\zeta+,\pi_k}$ started at vertex $\pi_k$,   which takes the horizontal step  $\cgeod{n+1}^{\zeta+,\pi_k}=\cgeod{n}^{\zeta+,\pi_k}+\evec_1$    whenever there is a tie
$B^{\zeta+}(\cgeod{n}^{\zeta+,\pi_k},\cgeod{n}^{\zeta+,\pi_k}+\evec_1)=B^{\zeta+}(\cgeod{n}^{\zeta+,\pi_k},\cgeod{n}^{\zeta+,\pi_k}+\evec_2)$ in the Busemann increments.   

Recall that $\cgeod{}^{\zeta+,\pi_k}$ is indexed so that  $\cgeod{n}^{\zeta+,\pi_k}\cdot(\evec_1+\evec_2)=\pi_n\cdot(\evec_1+\evec_2)$ for all $n\ge k$.   
By Theorem 4.3 in \cite{Geo-Ras-Sep-17-ptrf-2},  $\cgeod{}^{\zeta+,\pi_k}$ is directed into the segment $\Uset_{\zeta+}$.     Thus the two geodesics must separate eventually, and so for large enough $n$,  $\cgeod{n}^{\zeta+,\pi_k}\precneq \pi_n$.   The monotonicity of planar LPP increments  implies that 
\be\label{g619}    \Lpp_{x-\evec_1,\pi_n} -  \Lpp_{x,\pi_n}  \le   \Lpp_{x-\evec_1, \cgeod{n}^{\zeta+,\pi_k}} -  \Lpp_{x, \cgeod{n}^{\zeta+,\pi_k}}   \ee
for any $x\in\Z^2$ such that the LPP values are defined.   This so-called ``path crossing trick'' can be found for example in Lemma \ref{lm:B-order} below.
Now on $\Omega_0$ we have this upper bound, assuming $\Busgeo(x,\pi_k)<\infty$: 
\begin{align*}
\Busgeo(x-\evec_1,\pi_k) &= \lim_{n\to\infty} [\Lpp_{x-\evec_1,\pi_n} - \Lpp_{\pi_k, \pi_n}]  
  =   \lim_{n\to\infty} [ \Lpp_{x-\evec_1,\pi_n} -  \Lpp_{x,\pi_n}   + \Lpp_{x,\pi_n} -  \Lpp_{\pi_k, \pi_n} ]  \\
  &\overset{\eqref{g619}}\le  \varlimsup_{n\to\infty}\babs{ \Lpp_{x-\evec_1, \cgeod{n}^{\zeta+,\pi_k}} -  \Lpp_{x, \cgeod{n}^{\zeta+,\pi_k}}}    + \Busgeo(x,\pi_k) 
  \overset{\eqref{g602}}<\infty. 
\end{align*}
This completes the proof of  Theorem \ref{thm:g400}. 
\end{proof} 


 When $\w$ lies in the event  $\Omega_0$ constructed in Theorem \ref{thm:g400} and $\gamma$ is a nontrivial semi-infinite geodesic in the environment  $\w$, let $\Busgeo_\gamma(\w)$ denote the recovering cocycle constructed in Theorem \ref{thm:g400}. The two trivial semi-infinite geodesics in $\Geo_u^\w$ are $u + \Z_+e_1$ and $u + \Z_+e_2$.  All the other geodesics in $\Geo^\w_u$ are non-trivial.

\begin{lemma} \label{lm:B-order}   Let $\gamma\apreceq \pi$ be two nontrivial semi-infinite geodesics in an environment $\w\in\Omega_0$.   Then, $\Busgeo_\gamma\preceq \Busgeo_\pi$. 
    Consequently, if $\gamma\coal \pi$, then $\Busgeo_\gamma=\Busgeo_\pi$.
\end{lemma} 

\begin{proof} 
Take $x\in\Z^2$. Take $n$ large enough so that $x+e_1\le\gamma_n$, $x+e_1\le\pi_n$, and $\gamma_n\preceq\pi_n$.   Since $x$ is to the left of $\sigma^{x+e_1,\gamma_n}$ and $\pi_n$ to its right, $\sigma^{x+e_1,\gamma_n}$ must intersect $\sigma^{x,\pi_n}$. Let $z$ denote the first intersection point. Then 
\[\Lpp_{x,z}+\Lpp_{z,\gamma_n}\le\Lpp_{x,\gamma_n}\quad\text{and}\quad\Lpp_{x+e_1,z}+\Lpp_{z,\pi_n}\le\Lpp_{x+e_1,\pi_n}.\]
Add the two inequalities, use $\Lpp_{x,z}+\Lpp_{z,\pi_n}=\Lpp_{x,\pi_n}$ and $\Lpp_{x+e_1,z}+\Lpp_{z,\gamma_n}=\Lpp_{x+e_1,\gamma_n}$, and rearrange to get
\[\Lpp_{x,\pi_n}-\Lpp_{x+e_1,\pi_n}\le\Lpp_{x,\gamma_n}+\Lpp_{x+e_1,\gamma_n}.\]
Take $n\to\infty$ to get $\Busgeo_\pi(x,x+e_1)\le\Busgeo_\gamma(x,x+e_1)$. The inequality $\Busgeo_\pi(x,x+e_2)\ge\Busgeo_\gamma(x,x+e_2)$ is proved similarly.
\end{proof}


Given   a recovering cocycle $B$ in an environment $\w\in\Omega$, a \textit{$B$-geodesic} is an up-right path $\pi$, finite or infinite, whose steps obey minimal $B$-increments: 
$B(\pi_i, \pi_{i+1})=B(\pi_i, \pi_{i}+\evec_1)\wedge B(\pi_i, \pi_{i}+\evec_2)=\w_{\pi_i}$. Such a path is a geodesic. 
The $\evec_1$ tiebreaker geodesic  $\cgeod{}^{B,u,+}$ is the semi-infinite geodesic that starts at vertex $u$, follows minimal increments of $B$, and takes an $\evec_1$ step at a tie. It is the rightmost geodesic between any two its vertices \cite[Lemma 4.1]{Geo-Ras-Sep-17-ptrf-2}.  Analogously, $\cgeod{}^{B,u,-}$ is the semi-infinite $B$-geodesic from $u$ that takes an $\evec_2$ step at a tie. 

\begin{lemma} \label{lm:g428}  Let $B$ be a recovering cocycle in an  environment $\w\in\Omega_0$. Suppose there exists 
  a coalescing family  $\{\pi^u\}_{u\tsp\in\tsp\Z^2}$  of   semi-infinite $B$-geodesics  from all initial vertices $u\in\Z^2$. 
  Then $\Busgeo_{\pi^u}(\w)=B$ for every geodesic $\pi^u$ from this family. 
\end{lemma} 

\begin{proof}     
Given $x$ and $u$, let $\pi^u_N$ be the point where   $\pi^x$ and $\pi^u$ first coalesce.   Then for $n\ge N$, since we can follow $B$-geodesics, 
\begin{align*}
\Lpp_{x, \pi^u_n}-\Lpp_{u, \pi^u_n}
=B(x,\pi^u_n)-B(u,\pi^u_n) =B(x,u). 
\end{align*} 
Letting $n\to\infty$ gives  $\Busgeo
_{\pi^u}(x,u)=B(x,u)$. This and the cocycle property give $\Busgeo_{\pi^u}=B$.
\end{proof} 

In particular,  Lemma \ref{lm:g428} implies that if a recovering cocycle $B$ generates a coalescing family of cocycle geodesics, then any one of these geodesics is enough to identify $B$. 

\begin{lemma} \label{lm:Arho6}   
Let $\pi=\pi_{k:\infty}$ be a  nontrivial semi-infinite geodesic  in a fixed  environment $\w\in\Omega_0$ and  $\Busgeo_\pi(\w)$   the recovering cocycle constructed in Theorem \ref{thm:g400}. 
The geodesic $\pi$  is an $\Busgeo_\pi$-geodesic. It lies  between $\cgeod{}^{\pi_k,\Busgeo_\pi,-}$ and $\cgeod{}^{\pi_k,\Busgeo_\pi,+}$, the two geodesics generated by $\Busgeo_\pi$ that  resolve ties by taking $e_2$ and $e_1$ steps, respectively:
\be\label{Arho18}
\cgeod{}^{\pi_k,\Busgeo_\pi,-}\preceq \pi \preceq \cgeod{}^{\pi_k,\Busgeo_\pi,+}.\qedhere\ee
\end{lemma} 

\begin{proof} 
Since $\pi_{m:n}$ is a geodesic between $\pi_{m}$ and $\pi_{n}$ and goes through $\pi_{m+1}$, and then by recovery, 
\begin{align*}
    \Busgeo_\pi({\pi_m, \pi_{m+1}}) &=  
\lim_{n\to\infty}[ \Lpp_{\pi_m, \pi_n}-\Lpp_{\pi_{m+1}, \pi_n}] 
=  
\lim_{n\to\infty}[ \w_{\pi_m}+\Lpp_{\pi_{m+1}, \pi_n}-\Lpp_{\pi_{m+1}, \pi_n}]
\\
&=\w_{\pi_m} =  \Busgeo_\pi({\pi_m, \pi_{m}+\evec_1})\wedge  \Busgeo^\pi({\pi_m, \pi_{m}+\evec_2}). 
\end{align*} 
 This implies also  \eqref{Arho18}.
\end{proof}

\small
\bibliographystyle{plain}
\bibliography{masterbib.bib}

\end{document}